\documentclass{amsart}
\usepackage{amssymb}
\usepackage{enumitem}
\usepackage[utf8]{inputenc}
\usepackage{hyperref,fullpage,verbatim}
\usepackage[dvipsnames]{xcolor}

\usepackage{tikz}
\usepackage{pgfplots}
\usetikzlibrary{decorations.pathreplacing,angles,quotes}
\tikzstyle{my help lines}=[gray,
thick,dashed]
\usepackage{tkz-graph}
\tikzstyle{vertex}=[circle, draw, inner sep=0pt, minimum size=4pt]

\tikzstyle{vtx}=[circle, draw, inner sep=0pt, minimum size=8pt]

\usepackage{pgf-pie}
\newcommand{\Ratb}{\beta}
\newcommand{\Rata}{\alpha}

\theoremstyle{definition}
\newtheorem{theorem}{Theorem}[section]
\newtheorem{lemma}[theorem]{Lemma}

\newtheorem{proposition}[theorem]{Proposition}
\newtheorem{thm+def}[theorem]{Theorem and Definition}
\newtheorem{problem}[theorem]{Problem}

\newcommand{\tim}{\mathbf{s}}
\newcommand{\timtwo}{\mathbf{t}}

\newcommand{\C}{\mathbf{C}}
\newcommand{\R}{\mathbf{R}}
\newcommand{\Prob}{\mathbb{P}}
\newcommand{\E}{\mathbb{E}}
\newcommand{\D}{\mathbf{D}}
\newcommand{\Q}{\mathbf{Q}}
\newcommand{\CS}{$S_c^\ast$}

\newcommand{\CSB}{$S_c^B$}
\newcommand{\RSA}{$S_r^\ast$}
\newcommand{\RSB}{$S_r^B$}

\title{Tipsy cop and Tipsy robber: collisions of biased random walks on graphs}

\author{Pamela E.\ Harris}
\address[P.\ E.\ Harris]{Department of Mathematics, University of Wisconsin - Milwaukee, United States}
\email{\textcolor{blue}{\href{mailto:peharris@uwm.edu}{peharris@uwm.edu}}}
\thanks{P.~E.~Harris was supported by a Karen Uhlenbeck EDGE Fellowship.}

\author{Erik Insko}
\address[E.\ Insko]{Department of Mathematics, Florida Gulf Coast University, United States}
\email{\textcolor{blue}{\href{mailto:einsko@fgcu.edu}{einsko@fgcu.edu}}}

\author{Florian Lehner}
\address[F.\ Lehner]{Department of Mathematics, The University of Auckland, New Zealand}
\email{\textcolor{blue}{\href{mailto:mail@florian-lehner.net}{mail@florian-lehner.net}}}
\date{\today}

\begin{document}

\begin{abstract}
Introduced by Harris, Insko, Prieto Langarica, Stoisavljevic, and Sullivan, the \emph{tipsy cop and drunken robber} is a variant of the cop and robber game on graphs in which the robber simply moves randomly along the graph, while the cop moves directed towards the robber some fixed proportion of the time and randomly the remainder. 
In this article, we adopt a slightly different interpretation of tipsiness of the cop and robber where we assume that in any round of the game there are four possible outcomes: a sober cop move, a sober robber move, a tipsy (uniformly random) move by the cop, and a tipsy (uniformly random) move by the robber. 
We study this tipsy cop and tipsy robber game on the infinite grid graph and on certain families of infinite trees including $\delta$-regular trees  
and $\delta$-regular trees rooted to a $\Delta$-regular tree, where $\Delta \geq \delta$.   
Our main results analyze strategies for the cop and robber on these graphs. We conclude with some directions for further study.
\end{abstract}

\maketitle

\section{Introduction}

A \emph{random walk} on a graph $G$ is a stochastic process with random variables $(X_i)_{i \in \mathbb N}$ where $X_{i+1}$ is chosen among the neighbors of $X_i$ according to some predefined probability distribution. The study of random walks is a classical topic in discrete probability. 
For surveys on the study of random walks on graphs we refer the reader to \cite{survey1, survey2}. 
One question of interest is whether a given random walk is \emph{recurrent} (almost surely $X_i = X_0$ for infinitely many $i$) or \emph{transient} (almost surely  $X_i = X_0$ for finitely many $i$).
One of the first and most famous results in the study of recurrence and transience is P\'olya's theorem, stating that the simple random walk on $\mathbb Z^d$ is recurrent for $d \leq 2$, but transient for $d \geq 3$, or as Shizuo Kakutani succinctly summarized it in a talk: ``A drunk man will always find his way home, but a drunk bird may not.''\footnote{A joke by Shizuo Kakutani at a UCLA colloquium talk as attributed in Rick Durrett's book Probability: Theory and Examples.}

P\'olya describes his motivation for this result as follows, see \cite{Polya84}. He was taking a stroll in the woods when a young couple was walking in the same woods. They happened to cross paths so often that he felt embarrassed as he suspected they thought he was snooping, of which, he assures the reader that he was not.  
This incident caused him to ask how likely it is that two independent random walks meet infinitely often. In $\mathbb Z^d$ this reduces to the problem of recurrence or transience of a single random walk, which is precisely the outcome of P\'olya's theorem.

While recurrence and transience of random walks are certainly interesting properties to study, the story above suggests a different notion. Let us say that a random walk has the \emph{finite collision property} if two independent copies of it almost surely meet only finitely many times. As noted in \cite{KP04}, for simple random walks on vertex transitive graphs such as $\mathbb Z^d$ this notion is equivalent to transience, but there are examples of graphs where this is not the case. First examples of recurrent simple random walks with the finite collision property were studied by Krishnapur and Peres \cite{KP04}, for further results on the finite collision property of recurrent graphs see for instance \cite{BPS12,cwz08,chen2010,chen11,dgp18,hp15}.

There are also examples of transient random walks where two  or even arbitrarily many copies meet infinitely often \cite{chen16, dgp19}. The simplest such example is perhaps a biased random walk on $\mathbb Z$: the differences of two such walks yield a simple random walk on $\mathbb Z$ which is recurrent by P\'olya's theorem and thus returns to~$0$ infinitely often. One key property of these examples is that two independent copies almost surely `go to the same point at infinity' and move away from the starting vertex `at the same speed.'

In the examples mentioned above it was always assumed that the two random walks are independent and that their distributions are identical; in the present paper we drop both of these assumptions and study collisions between two biased random walks, where (roughly speaking) the first walk has a bias to reduce the distance between the positions of the two walks, whereas the second walk has a bias to increase this distance. 

As the title suggests, this situation can be modeled by a pursuit-evasion game similar to the cop-and-robber game, a two-player game introduced independently by Quilliot \cite{Q86} and by Nowakowski and Winkler \cite{NW83} in the 1980s. In this game, two players (called cop and robber) alternate in moving their playing pieces from vertex to vertex along the edges of a given graph. The cop wins if after finitely many steps both the cop and robber are at the same vertex. The monograph \cite{BN11} by Bonato and Nowakowski offers an extensive treatment of this game and numerous variants thereof.
 
A variant which is relevant to this paper is the cop and drunken robber game introduced by 
Kehagias and Pra\l at \cite{KP2012}, see also \cite{KW14}.
% and recently studied on the $n$-dimensional grid graph \cite{JKST21}. 
In this variant, the cop may follow some strategy, while the robber moves `drunkenly' according to a random walk on the graph. 
More recently a variant called the \emph{tipsy cop and drunken robber} was introduced \cite{BCI21,HIPSS20}, where in addition to the robber moving randomly, the cop only follows a strategy on some steps and performs random moves otherwise.

In the current work, we consider the case in which both players can follow a strategy on some (randomly selected) moves and make uniform random moves otherwise. We call this game the \emph{tipsy cop and tipsy robber} game. Roughly speaking, the tipsiness of each player governs the proportion of random moves of this player; we make this precise in Section \ref{section:background}. We note that if both players have tipsiness $1$, then all moves are random and thus the question of who wins the game is equivalent to the question whether two independent random walks almost surely collide. The other extreme case, when both players have tipsiness $0$, resembles the classical cop and robber game, but the two are not quite equivalent since in our game we choose randomly which player gets to move, and hence the moves of the two players do not necessarily alternate.

As in all two-player games, one question of interest is whether one of the two players has a winning strategy. Let us call a cop strategy winning if it almost surely leads to a win for the cop, and let us call a robber strategy winning if it leads to a win for robber with positive probability. The reason for this asymmetry in the definition of winning strategies is that (due to the randomness involved) there is always a positive probability that the cop wins the game before either of the two players gets to make a non-random move. It is worth noting that this observation also shows that on a finite graph, the cop will almost surely win the game, which is why we focus our attention on infinite graphs.

In Section~\ref{section:grid_strategy}, we study the game on the infinite Cartesian grid $\mathbb Z^2$. We show that the robber has a winning strategy if and only if the robber is less tipsy than the cop. 
In fact, if the robber is less tipsy, then it is not hard to show that the precise strategy does not matter so long as the distance between the robber and the cop increases on every non-random robber move.  
Perhaps surprisingly, the cop's winning strategies are more sensitive: there are strategies where each non-random cop move decreases the distance between the cop and the robber that still allow the robber to win with positive probability despite being more tipsy than the cop.

In Section~\ref{section:tree_strategy}, we analyze the game on certain families of infinite trees.
It is not hard to see that for regular trees, the question of who has a winning strategy boils down to a simple application of the gambler's ruin problem, but there are also trees exhibiting more interesting behavior.  
We consider the game on a family of infinite trees, $ \{ X(\Delta,\delta) \} $, where the tree $X(\Delta, \delta)$ is created by starting with a $\Delta$-regular base tree and attaching a copy of a $\delta$-regular tree to each node in this base tree.
These trees provide examples where the robber's optimal strategy does not necessarily increase the distance between the players: under certain conditions the robber's optimal strategy may be to backtrack toward the cop in order to reach the base tree, where the number of possible escape routes is higher.   
Interestingly, these conditions not only depend on the values $\Delta$ and $\delta$, but also on the tipsiness parameters.

We conclude the article in Section \ref{section:future} by providing some possible directions for further study.

\section{Background and general set up}\label{section:background}
\subsection{Markov chains and random walks}

A Markov chain with state space $V$ is a random process $(X_{\tim})_{\tim \in \mathbb N}$ where each $X_{\tim}$ takes values in some space $V$ such that
\[
\Prob[X_{\tim+1} = x \mid X_0 = x_0, X_1=x_1, \dots, X_{\tim} = x_{\tim}] = \Prob[X_{\tim+1} = x \mid X_{\tim} = x_{\tim}]
\]
whenever the event we condition on has positive probability. For a thorough introduction to the topic of denumerable Markov chains see \cite{Woess09}; in this section we briefly recall some well known facts. Let us say that the Markov chain \emph{is at $x$ at time step $\tim$} if $X_{\tim} = x$, and let us denote the \emph{transition probability} $\Prob[X_{\tim+1} = y \mid X_{\tim} = x]$ by $p_{xy;\tim}$, or sometimes (for readability reasons) by $p_{x,y; \tim}$. We also write $p_{xy,\tim}^n$ for the $n$-step transition probabilities, that is, $p_{xy,\tim}^n = \Prob[X_{\tim+n} = y \mid X_{\tim} = x]$. 
In order to make a Markov chain well defined, we also need to specify a distribution of the first random variable $X_0$ in the process. For the most part, we will assume that this distribution is concentrated on a single element of $V$. As usual, for an event $A$ we will denote by $\Prob_v[A] = \Prob[A \mid X_0=v]$, and  by $\E_v[A] = \E[A \mid X_0=v]$.

A Markov chain is called \emph{time homogeneous} if $p_{xy;\tim}$ does not depend on $\tim$; in this case we write $p_{xy}$ instead of $p_{xy;\tim}$, and $p_{xy}^n$ instead of $p_{xy;\tim}^n$. Most Markov chains that appear in this paper have this property. A Markov chain is called \emph{irreducible} if for every pair $x,y$ of elements of $V$ and every $\tim \in \mathbb N$ there is some $n$ such that the probability $p_{xy;\tim}^n$ that $\Prob[X_{\tim+n}=y\mid X_{\tim}=x] > 0$.

Let $G$ be a graph with vertex set $V(G)$ and edge set $E(G)$. A random walk on $G$ is a Markov chain with state space $V = V(G)$, where $p_{xy,\tim}$ is only allowed to be non-zero if $xy$ is an edge of $G$. When speaking about random walks on $\mathbb N$ or $\mathbb Z$, we always assume that edges are between elements whose difference is $1$. If $p_{xy} = \frac{1}{\deg x}$ for every pair $x,y$ of adjacent vertices, then we speak of a simple random walk; in this case the random walk is equally likely to transition from $x$ to any of its neighbors. Another well-established way to define transition probabilities for a random walk is by assigning a weight $w(e) >0$ to each edge $e$ of the graph. The transition probability from $x$ to $y$ in this case is given by 
\begin{align}\label{weights}
p_{xy} = \frac{w(xy)}{\sum_{z\sim x} w(xz)},
\end{align}
in other words, the probability of moving from $x$ to any of its neighbors is proportional to the weight of the corresponding edge. The simple random walk is a special case of this, where all edge weights are equal. We note that any random walk on a connected graph defined by edge weights as in \eqref{weights} is time homogeneous and irreducible.

For every $v \in V$, let $T_{v}^+ = \inf\{\tim > 0 \mid X_{\tim}=v\}$ be the random variable counting the number of steps it takes the Markov chain to reach $v$ (or return to $v$ if $X_0 = v$). We call $v \in V$ \emph{recurrent}, if $\Prob_{v}[T_{v}^+< \infty] = 1$, and \emph{transient} otherwise. If $v$ is recurrent, then we call it \emph{positive recurrent} if $\E_v[T_{v}^+]< \infty$, and \emph{null recurrent} if $\E_v[T_{v}^+]= \infty$. We call a Markov chain positive recurrent, null recurrent, or transient if all of its states have this property.

The following theorem which is \cite[Theorems 3.2 and 3.9]{Woess09} implies that a time homogeneous, irreducible Markov chain is either positive recurrent, or null recurrent, or transient. 

\begin{theorem}
\label{thm:mc-recurrent-transient}
Let $X_{\tim}$ be a time homogeneous, irreducible Markov chain with state space $V$, and let $v,x,y \in V$. The following (mutually exclusive) statements hold:
\begin{enumerate}
    \item The state $v$ is recurrent if and only if $\Prob_{x}[T_{y}^+< \infty] = 1$, which is the case if and only if $\Prob_{x}[X_{\tim} = y \text{ for infinitely many }\tim] = 1$. Moreover,
    \begin{enumerate}
        \item $v$ is positive recurrent if and only if $\E_x[T_{y}^+]< \infty$, and
        \item $v$ is null recurrent if and only if $\E_x[T_{y}^+]= \infty$.
    \end{enumerate}
    \item The state $v$ is transient if and only if $\Prob_x[T_{y}^+< \infty] < 1$, which is the case if and only if $\Prob_{x}[X_{\tim} = y \text{ for infinitely many }\tim] = 0$.
\end{enumerate}
\end{theorem}

In the remainder of this section we provide some well-known results about random walks and Markov chains. Most of these results can be deduced from results contained in textbook on these topics, see for instance \cite{LyonsPeres, woess08,Woess09}.

The first result we mention is Pólya's theorem on recurrence and transience of simple random walks on~$\mathbb Z^d$.

\begin{theorem}
The simple random walk on $\mathbb Z^d$ is null recurrent if $d \in \{1,2\}$ and transient if $d \geq 3$.
\end{theorem}

The next result is \cite[Exercise 2.1 (f)]{LyonsPeres}.

\begin{theorem}
\label{thm:rw-edgeweights}
A random walk on a graph defined by edge weights as in \eqref{weights} is positive recurrent if and only if the sum of the edge weights is finite.
\end{theorem}

The next result is a direct consequence of \cite[Theorem 2.16]{LyonsPeres}.

\begin{theorem}%[Consequence of Rayleigh's monotonicity principle] 
\label{thm:RMP}
Let $w$ and $w'$ be weight functions on $E(G)$ and assume that there are constants $c_1>0$ and $c_2$ such that $c_1 w(e) \leq w'(e) \leq c_2 w(e)$ for every $e \in E(G)$. The random walk defined by the weights $w$ is recurrent if and only if the simple random walk $w'$ on $G$ is recurrent.
\end{theorem}

The last few results concern biased random walks on $\mathbb N$. These are random walks where the transition probability $p_{n,n+1}$ is $p$, the transition probability $p_{n,n-1}$ is $q \leq 1-p$,  the transition probability $p_{n,n}$ is $1-p-q$ for every $n >0$, and $p_{0,1}=1$. The case $q=1-p$ where $p_{n,n}=0$ is of particular interest. Variants of the following well-known result can be found in many textbooks on probability theory, the case $p + q=1$ which (by a coupling argument) implies the general case can for instance be found in \cite[Theorem 5.9, Example~5.10]{Woess09}.

\begin{theorem}[Gambler's ruin]
\label{thm:gamblersruin}
Let $X_{\tim}$ be a random walk on $\mathbb N$ with transition probabilities $p_{0,1}>0$, $p_{0,0} = 1-p_{0,1}$, and $p_{n,n+1} = p > 0$, $p_{n,n-1}=q>0$ and  $p_{n,n}=1-p-q$ for $n > 0$.
\begin{itemize}
    \item If $p> q$, then the random walk is transient.
    \item If $p < q$, then the random walk is positive recurrent.
    \item If $p=q$, then the random walk is null recurrent.
\end{itemize}
\end{theorem}

\begin{proposition}
\label{prp:occupationmeasure}
\label{prp:zerobeforeincrease}
Let $X_{\tim}$ be a random walk on $\mathbb N$ with the same transition probabilities as in Theorem \ref{thm:gamblersruin}, and assume that $p<q$, and that $p_{0,0}>0$. 
\begin{enumerate}
    \item There is a probability measure $\nu$ on $\mathbb Z$ such that $\lim_{\tim \to \infty} \Prob[X_{\tim} = x] = \nu(x)$ for every $x$.
    \item Almost surely $\lim_{\tim \to \infty} \frac{1}{\tim} |\{\timtwo < \tim \mid X_{\timtwo} = x\} = \nu(x)$.
    \item There is an absolute constant $c$ such that $\Prob_n[T_0^+ < T_{n}^+]>c$.
\end{enumerate}
\end{proposition}
\begin{proof}
The first part follows from \cite[Theorem 3.48]{Woess09}, the second part from \cite[Theorem 3.55]{Woess09}.
For the third part, note that a random walk $Y_{\tim}$ with transition probabilities $p_{n,n+1} = q$, $p_{n,n-1}=p$ is transient and therefore has a positive probability of never returning to $n$. In particular, for such a random walk $\Prob_n[T_{2n}^+ < T_{n}^+] \geq \Prob_n[T_{n}^+ = \infty] =: c$, and the result follows by the obvious coupling of $X_{\tim}$ and $Y_{\tim}$, that is, the coupling where $X_{\tim}-X_{\tim-1}= Y_{\tim-1} - Y_{\tim}$ until $X_{\tim} =0$ for the first time, and the two processes are independent afterwards.
\end{proof}

Theorem \ref{thm:gamblersruin} also holds for reasonably well-behaved inhomogeneous random walks.

\begin{proposition}
\label{prp:inhomogeneous_recurrent}
Let $X_{\tim}$ be a random walk on $\mathbb N$ whose transition probabilities are $p_{n,n+1;\tim} = p_{\tim}$, $p_{n,n-1;\tim} = q_{\tim}$ and  $p_{n,n;\tim} = 1-p_\tim-q_{\tim}$ for all $n > 0$, and $p_{0,1;\tim} > 0$. Let $Y_{\tim}$ be a random walk on $\mathbb N$ with constant transition probabilities $p_{n,n+1;\tim} = p$, $p_{n,n-1;\tim} = q$, and $p_{n,n;\tim}=1-p-q$ for all $n > 0$, and $p_{0,1} = 1$. Assume that $X_0$ and $Y_0$ have the same distribution.
\begin{enumerate}
    \item If $p_{\tim} \leq p$ and $q_{\tim} \geq q$, for all $\tim$, then there is a coupling of $X_\tim$ and $Y_\tim$ for which $X_\tim \leq Y_\tim+1$. In particular $\Prob_x[X_\tim \leq y+1] \geq \Prob_x[Y_\tim \leq y]$ for all $x,y \in \mathbb N$ and all $\tim$, and if $Y_{\tim}$ is (positive) recurrent, then $X_{\tim}$ is also (positive) recurrent.
    \item If $p_{\tim} \geq p$, $q_{\tim} \leq q$ and $p_{0,1;\tim} =1$ for all $\tim$, then there is a coupling of $X_\tim$ and $Y_\tim$ for which $X_\tim \geq Y_\tim-1$. In particular $\Prob_x[X_\tim \leq y-1] \leq \Prob_x[Y_\tim \leq y]$ for all $x,y \in \mathbb N$ and all $\tim$. If $Y_{\tim}$ is transient, then $X_{\tim}$ is also transient.
\end{enumerate}
%If $p_{\tim}\leq q_{\tim}$ for all $\tim \in \mathbb N$, then almost surely $X_{\tim} = 0$ for infinitely many $\tim$. Conversely, if there is some $\epsilon > 0$ such that $p_{\tim} \geq q_{\tim}+\epsilon$ for all $\tim \in \mathbb N$, then there is a positive probability that there is no $\tim > 0$ such that $X_{\tim} = 0$.
\end{proposition}
\begin{proof}
    Note that given $X_0$, we can view $X_\tim$ as a function a sequence $U_\tim$ of i.i.d\ uniform random variables on $[0,1)$ as follows. If $X_\tim = n$, then
    \[
     X_{\tim+1} =
     \begin{cases}
         n+1 &\text{if }U_{\tim} < p_{n,n+1;\tim}\\
         n &\text{if } p_{n,n+1;\tim} \leq U_{\tim;\tim} < 1- p_{n,n-1;\tim}\\
         n-1 &\text{otherwise,}
     \end{cases}
    \]
    where we set $p_{0,-1;\tim} = 0$.
    Similarly, given $Y_0$, we can view $Y_\tim$ as a function a sequence $U_\tim$ of i.i.d\ uniform random variables on $[0,1)$.

    Consider a coupling of $X_\tim$ and $Y_\tim$ with $X_0=Y_0$ in which $X_\tim$ and $Y_\tim$ are functions of the same sequence $U_\tim$. It can be seen by induction on $\tim$ that this coupling satisfies the claimed properties in both cases. 
\end{proof}

Finally, we note that more precise statements about the behavior of the random walk from Theorem \ref{thm:gamblersruin} can be made in both the positive recurrent and the transient case. 
Note that random walks on $\mathbb Z$ can be interpreted as sums of random variables each taking values in $\{-1,0,1\}$. If these random variables are i.i.d., then their sum is well-understood. 

We will need the following concentration result.

\begin{theorem}
\label{thm:concentration}
Let $X_i$ be a sequence of i.i.d.\ random variables each of which is concentrated on a finite interval. Let $S_n = \sum_{i \leq n} X_i$, let $M_n = \min_{k\leq n} S_k$, and let $\mu = \E[X_i]$. There is a constant $c > 0$  such that  for every $\epsilon > 0$ 
\[
    \Prob[|S_n-n\mu| >  n \epsilon] <  e^{-c \epsilon^2 n}.
\]
If $\mu < 0$, then there are constants $a,b > 0$ such that for every $\epsilon$
\[
    \Prob[M_n <  (\mu - \epsilon)n] <  a e^{-b \epsilon^2 n}.
\]
\end{theorem}

\begin{proof}
The first part follows from the Azuma-Hoeffding inequality \cite[Theorem 13.2]{LyonsPeres} applied to the two random variables $X_i - \mu$ and $\mu - X_i$.

For the second part note that
\[
\Prob[\min_{k\leq n} S_k < (\mu - \epsilon)n] \leq \sum_{k=0}^n \Prob[S_k - \mu k < \mu (n - k) - n\epsilon] \leq \sum_{k=0}^n \Prob[|S_k - \mu k| > n\epsilon].
\]
Applying the first part to the summands gives
\[
\Prob[\min_{k\leq n} S_k < (\mu - \epsilon)n] < \sum_{k=0}^n e^{-ck (n\epsilon / k)^2} \leq 
\sum_{k=0}^n e^{-c n \epsilon^2} = n e^{-cn\epsilon^2},
\]
and the desired result follows by taking $b<c$ and $a$ large enough.
\end{proof}

\begin{lemma}
\label{lem:transientonN}
Let $X_{\tim}$ be a random walk on $\mathbb N$ with $p_{0,0}<1$, $p_{0,1}= 1-p_{0,0}$, and  $p_{n,n+1} = p$,  $p_{n,n-1} = q < p$, and $p_{n,n} = 1- p-q$ for every $n > 0$. 

Then there are positive constants $k$, $a$, and $b$ such that $\Prob_n[X_{\tim}<n+k\tim] < a \cdot e^{-b\tim}$ for every $n \geq 0$. Moreover, $\Prob_n[\exists \tim_0 \colon \forall \tim \geq \tim_0 \colon X_{\tim} \geq n+k\tim] =1$, and for every $\epsilon> 0$ we can find $n_0 \in \mathbb N$ such that $\Prob_n[\forall \tim \geq 0 \colon X_{\tim} \geq n-n_0+k\tim] \geq 1-\epsilon$.
\end{lemma}

\begin{proof}
Let $Y_{\tim}$ be a random walk on $\mathbb Z$ with transition probabilities $p_{n,n+1} = p$ and $p_{n,n-1} = q$. Consider a coupling of $X_{\tim}$ and $Y_{\tim}$ such that $Y_0 = X_0$ and $Y_{\tim+1}-Y_{\tim} = X_{\tim+1} -  X_{\tim}$, whenever $X_{\tim} >0$. Let $R_{\tim} = |\{\timtwo \leq \tim \colon X_{\timtwo} = 0\}|$, and let $R_{\infty} = |\{\timtwo \in \mathbb N \colon X_{\timtwo} = 0\}|$.

Clearly, $Y_{\tim} - 2 R_{\tim} \leq X_{\tim} $. The differences $Y_{\tim} - Y_{\tim-1}$ are i.i.d.\ and Theorem \ref{thm:concentration} implies that there is some $k_0$ and $c$ such that $\Prob_n[Y_{\tim}<k_0\tim] < e^{-c\tim}$. Further note that 
%$\Prob_1[\forall \tim > 0\colon X_{\tim} > 1] >0$ by Theorem \ref{thm:gamblersruin}. Hence 
$X_{\tim}$ is transient and thus $\Prob_0[T_0^+ < \infty] = p < 1$. This implies that there is some constant $c'$ such that $\Prob_0[R_{\infty} \geq x] =(1-p)^x \leq e^{-c'x}$ for every $x$.

Hence for every $\tim > \tim_0$ we have
\begin{align*}
  \Prob_n[X_{\tim}<n+k_0\tim/3] 
  &\leq \Prob_n[Y_{\tim}<k_0\tim] + \Prob_n[R_{\tim}>k_0\tim/3]\\
  &\leq \Prob_n[Y_{\tim}<k_0\tim] + \Prob_0[R_{\infty}>k_0\tim/3]\\
  &\leq e^{-c\tim} + e^{-c'\tim/3}.
\end{align*}
Taking $k = k_0/3$, $b = \min (c,c'/3)$, and $a$ such that $a\cdot e^{-b\tim} > 1$ for $\tim < \tim_0$ yields the desired result.

The statement $\Prob_n[\exists \tim_0 \colon \forall \tim \geq \tim_0 \colon X_{\tim} \geq n+k\tim] =1$ follows from the Borel--Cantelli lemma because $\sum_{\tim \geq 0} a \cdot e^{-b\tim} < \infty$. Finally note that $\Prob_n[\forall \tim \geq 0 \colon X_{\tim} \geq n-n_0+k\tim] \geq \Prob_n[\forall \tim \geq n_0 \colon X_{\tim} \geq n+k\tim]$, and the latter probability tends to 1 as $n_0$ tends to infinity (because the probability of the complementary event is bounded above by the tail of a convergent series).
\end{proof}

On the other hand, in the positive recurrent case, one can show that the random walk will spend a large proportion of steps `close to $0$'. 

\begin{lemma}
\label{lem:positiverecurrentonN}
Let $X_{\tim}$ be a random walk on $\mathbb N$ with $p_{0,0}<1$, $p_{0,1}= 1-p_{0,0}$, and  $p_{n,n+1} = p$,  $p_{n,n-1} = q > p$, and $p_{n,n} = 1- p-q$ for every $n > 0$. Let $U_{\tim} = |\{\timtwo \leq \tim \colon X_{\timtwo} = X_{\timtwo - 1} = 0\}$.

Then for every $\epsilon > 0$ there are positive constants $a$ and $b$ such that
\[
\Prob_n\left[\left(\frac{(q-p)(1-p_{0,1})}{p_{0,1}+q-p} -\epsilon\right)\tim < U_{\tim} < \left(\frac{(q-p)(1-p_{0,1})}{p_{0,1}+q-p} +\epsilon\right)\tim\right]> 1-ae^{-b\tim}.
\]
\end{lemma}

\begin{proof}
Let $Y_{\tim}$ be a random walk with $\Prob[Y_0=0]=1$ and transition probabilities $p_{n,n}^Y = p_{0,0}$ and  $p_{n,n+1}^Y= p_{0,1}$, and let  $Z_{\tim}$ a random walk with $\Prob[Z_0=0]=1$ and transition probabilities $p_{n,n+1}^Z = p$,  $p_{n,n-1}^Z= q$, and  $p_{n,n}^Z= 1-p-q$. 

Consider the following coupling of $X_{\tim}$ with $Y_{\tim}$ and $Z_{\tim}$. Denote by $f_y(\tim) = |\{\timtwo < \tim \colon X_{\timtwo}=0\}|$ and by $f_z(\tim) = |\{\timtwo < \tim \colon X_{\timtwo}>0\}|$. If $X_{\tim} = 0$, then $X_{\tim+1}-X_{\tim} = Y_{f_y(\tim)+1}-Y_{f_y(\tim)}$; if $X_{\tim} > 0$, then $X_{\tim+1}-X_{\tim} = Z_{f_z(\tim)+1}-Z_{f_z(\tim)}$.

We note that by Theorem \ref{thm:concentration}, for every $\epsilon >0$ there are positive constants $a,b$ such that the following holds:
\begin{align}
    \Prob[(p_{0,1}-\epsilon)f_y(\tim) <Y_{f_y(\tim)} < (p_{0,1}+\epsilon)f_y(\tim)] &> 1-ae^{-bf_y(\tim)},\label{eq:boundy}\\
    \Prob[(p-q-\epsilon)f_z(\tim) <Z_{f_z(\tim)} < (p-q+\epsilon)f_z(\tim)] &> 1-ae^{-bf_z(\tim)},\label{eq:boundz}\\
    \Prob[(p-q-\epsilon)f_z(\tim) < \min_{\timtwo < \tim}Z_{f_z(\timtwo)}] &> 1-ae^{-bf_z(\tim)}.\label{eq:boundminz}
\end{align}

We further note that $Y_{f_z(\tim)}$ reaches a new maximum when $X_{\tim-1}=0$ and $X_{\tim}=1$, and $Z_{f_z(\tim)}$ reaches a new minimum when $X_{\tim-1}=1$ and $X_{\tim}=0$. Hence
\begin{equation}
\min_{\timtwo \leq \tim }Z_{f_z(\timtwo)} + X_0 = - Y_{f_y(\tim)}\leq  Z_{f_z(\tim)} + X_0. \label{eq:yandz}
\end{equation}
Thus for every $\epsilon > 0$ there is some positive constant $c$ (depending on $\epsilon$) and $\tim_0$ (depending on $X_0$ and $\epsilon$) such that for every $\tim > \tim_0$
\begin{equation*}
\Prob\left[\left(\frac{q-p}{p_{0,1}} -\epsilon\right)f_z(\tim) < f_y(\tim) < \left(\frac{q-p}{p_{0,1}} +\epsilon\right)f_z(\tim)\right] > 1-ae^{-bf_y(\tim)} -ae^{-bf_z(\tim)},
\end{equation*}
and since $\tim = f_y(\tim) + f_z(\tim)$ this is equivalent to
\begin{equation}
\Prob\left[\left(\frac{q-p}{p_{0,1}+q-p} -\epsilon\right)\tim < f_y(\tim) < \left(\frac{q-p}{p_{0,1}+q-p} +\epsilon\right)\tim\right] > 1-ae^{-bf_y(\tim)} -ae^{-b(\tim - f_y(\tim))}.\label{eq:boundfy}
\end{equation}

We would like to bound the right side of this inequality by something that only depends on $\tim$. To this end, note that $f_z(\tim) \geq - \min_{\timtwo < \tim}Z_{f_z(\timtwo)}$. Combining this with \eqref{eq:boundy} and \eqref{eq:yandz}, for  suitable positive constants $a$ and $b$ we have that
\begin{align*}
   \Prob\left[ f_z(\tim) - X_0 > \frac{p_{0,1} - \epsilon}{2} \tim \mid f_y(\tim) \geq \frac{\tim}2\right] \geq 
    \Prob\left[-\min_{\timtwo < \tim}Z_{f_z(\timtwo)} - X_0 > \frac{p_{0,1} - \epsilon}{2} \tim \mid f_y(\tim) \geq \frac{\tim}2\right] > 1 - ae^{-b\frac{\tim}2}.
\end{align*}

Similarly, note that ${f_y(\tim)}\geq Y_{f_y(\tim)}$. Together with \eqref{eq:boundz} and \eqref{eq:yandz} we thus obtain
\begin{align*}
    \Prob\left[{f_y(\tim)} + X_0 > \frac{(q-p+\epsilon)}{2} \tim \mid f_z(\tim) \geq \frac{\tim}2\right] \geq 
    \Prob\left[-Y_{f_y(\tim)} - X_0 < \frac{(p-q-\epsilon)}{2} \tim \mid f_z(\tim) \geq \frac{\tim}2\right] > 1 - ae^{-b\frac{\tim}2}.
\end{align*}

Since $\tim = f_y(\tim) + f_z(\tim)$ one of the two must be at least $\frac{\tim}{2}$; we conclude that there are positive constants $a$, $b$, and $\alpha$ such that 
\begin{equation*}
    \Prob[\min (f_y(\tim),f_z(\tim)) \leq \alpha \tim] <  ae^{-b\tim}.
\end{equation*}
Combining this with \eqref{eq:boundfy}, we obtain that there are positive constants $a$ and $b$ (depending on $X_0$) such that for every $\tim$
\begin{equation*}
\Prob\left[\left(\frac{q-p}{p_{0,1}+q-p} -\epsilon\right)\tim < f_y(\tim) < \left(\frac{q-p}{p_{0,1}+q-p} +\epsilon\right)\tim\right]> 1-ae^{-b\tim}.
\end{equation*}
Equation \eqref{eq:boundy} together with the observation that $U_\tim = f_y(\tim) - Y_{f_y(\tim)}$ yields the desired result.
\end{proof}

\subsection{The tipsy cop and robber game}

In this paper, we will study pairs $\C_{\tim},\R_{\tim}$ of random walks, where the transition probabilities in each step depend on the current relative positions of the random walks. Roughly speaking, the random walk $\C_{\tim}$ should be biased to maximize the probability that the two walks meet, whereas the random walk $\R_{\tim}$ should aim to minimize the meeting probability.

More precisely, let $G$ be a connected, locally finite graph. We have four parameters $p_c^s$, $p_c^t$, $p_r^s$, and $p_r^t$ satisfying the condition $ p_c^s+p_c^t+p_r^s+p_r^t =1$, and a pair of \emph{strategy functions}  $S_c$ and $S_r$, which map each pair in $V(G) \times V(G)$ to a probability distribution where $S_c(u,v)$ is concentrated on the neighborhood of $u$, and $S_r(u,v)$ is concentrated on the neighborhood of $v$. 
%By a slight abuse of notation, we write $S_c(u,v)$ and $S_r(u,v)$ for a random vertex chosen according to these distributions; this is justified by the fact that 
We note that most strategy functions we consider are deterministic, that is, they assign probability 1 to one vertex and probability 0 to all other vertices. 

Given the above parameters, we define transition probabilities
\[
    p_{(u,v),(u',v')} =
    \begin{cases}
        p_c^s \cdot S_c(u,v) (u') + p_c^t \cdot \frac{1}{\deg u} & \text{if }u' \in N(u) \text{ and } v'=v, \\
        p_r^s \cdot S_r(u,v) (v') + p_c^t \cdot \frac{1}{\deg v} & \text{if }u' =u \text{ and } v' \in N(v), \\
        0 & \text{otherwise,}
    \end{cases}
\]
where $S_c(u,v)(x)$ and $S_r(u,v)(x)$ denote the probability of the vertex $x$ with respect to the distributions $S_c(u,v)$ and $S_r(u,v)$, respectively. 
We note that this defines a Markov chain with state space $V(G) \times V(G)$ whose transition probabilities depend on $p_c^s$, $p_c^t$, $p_r^s$, and $p_r^t$, as well as on the strategy functions. We will be interested in the probability of the event that there is some time step $\tim\in \mathbb N$ for which $\C_{\tim}$ and $\R_{\tim}$ coincide.

It will be convenient to consider $(\C_{\tim},\R_{\tim})$ as a function of $(\C_{0},\R_{0})$ together with an i.i.d.\ process. Let $M_{\tim}$ and $X_{\tim}$ be i.i.d.\ random variables, where $M_{\tim}$ is chosen from $\{\mathrm{cs} ,\mathrm{rs} ,\mathrm{ct} ,\mathrm{rt} \}$ with $\Prob[M_{\tim} = xy] = p_x^y$, and $X_{\tim}$ is uniform on the unit interval $[0,1)$. Given an initial position $\Q_{0}$ we define $\Q_{\tim}$ as functions of $M_{\tim}$ and $X_{\tim}$ as follows.
Let $\Q_{\tim-1}=(u,v)$, and fix an enumeration $u_1, \dots , u_{\deg u}$ of the neighbors of $u$ and $v_1, \dots, v_{\deg v}$ of the neighbors of $v$. We note that we allow these enumerations to depend on the pair $(u,v)$. In particular, the order of the neighbors of $u$ with respect to the pair $(u,v)$ might not be the same as the order with respect to the pair $(u,v')$. For each $i$ let 
$x_{c,i} = \sum_{j=1}^i S_c(u,v)(u_i)$, and let 
$x_{r,i} = \sum_{j=1}^i S_r(u,v)(v_i)$. With these definitions we can define
\[
    \Q_{\tim} = 
    \begin{cases}
        (u_i,v) & \text{if } (M_{\tim} = \mathrm{cs} \text { and } x_{c,i-1} \leq X_{\tim} < x_{c,i}) \text{ or }
        (M_{\tim} = \mathrm{ct}  \text { and } \frac{i-1}{\deg u} \leq X_{\tim} < \frac{i}{\deg u}),
        \\
        (u,v_i) & \text{if } (M_{\tim} = \mathrm{rs}  \text { and } x_{r,i-1} \leq X_{\tim} < x_{r,i}) \text{ or }
        (M_{\tim} = \mathrm{rt}  \text { and } \frac{i-1}{\deg v} \leq X_{\tim} < \frac{i}{\deg v}). 
    \end{cases}
\]
Clearly $\Prob[\Q_{\tim} = (u',v')\mid \Q_{\tim-1} = (u,v)] = p_{(u,v),(u',v')}$, where $p_{(u,v),(u',v')}$ is the transition probability defined above, and $\Prob[\Q_{\tim} = q_{\tim}\mid \Q_{\tim-1} = q_{\tim-1}] = \Prob[\Q_{\tim} = q_{\tim}\mid \Q_{\tim-1} = q_{\tim-1},  \dots, \Q_{0} = q_0]$. Hence if $\Q_0$ and $(\C_0,\R_0)$ have the same distribution, then so do the random processes $\Q_{\tim}$ and $(\C_{\tim},\R_{\tim})$. 

Throughout the rest of this paper we will identify $(\C_{\tim},\R_{\tim})$ with the process $\Q_{\tim}$ defined above.
Besides being able to use the auxiliary random variables $M_{\tim}$ and $X_{\tim}$ in our proofs, this has the additional advantage of giving a natural coupling of the processes $(\C_{\tim},\R_{\tim})$ and $(\C_{\tim}',\R_{\tim}')$ on the same graph for which the parameters $p_c^s$, $p_c^t$, $p_r^s$, and $p_r^t$ coincide, but the strategy functions differ. Simply consider both of them as functions of the same sequence of hidden variables $M_{\tim}$ and $X_{\tim}$ with the same initial positions. We will refer to this as the \emph{standard coupling} of  $(\C_{\tim},\R_{\tim})$ and $(\C_{\tim}',\R_{\tim}')$.

We will think of the process defined above as a pursuit-evasion game on a graph which we call the tipsy cop and robber game. This game is played between two players, the cop and the robber, each of whom controls a playing piece. A move by either player consists of taking their respective playing piece and moving it to an adjacent vertex. 
In every round, $M_{\tim}$ can be thought of as the outcome of a spinner wheel with four possibilities (or a biased 4-sided die) to determine which kind of move will happen in this round. The four outcomes of $M_{\tim}$ coincide with the four transition options above: (1) the cop can make a move of their choosing, potentially involving some randomness, (2) the robber can make a move of their choosing, potentially involving some randomness, (3) the cop has to make a move chosen uniformly at random, (4) the robber has to make a move chosen uniformly at random.  An example of one such spinner is depicted in Figure~\ref{fig:spinner}. 
\begin{figure}
\begin{tabular}{cc}
\begin{tikzpicture}[scale = 1.65]
\pie[
  color = {yellow, green, cyan, red},radius = .75,
  /tikz/nodes={text opacity=1,overlay}
] {25/$p_r^s$,30/$p_c^s$,20/$p_c^t$,25/$p_r^t$}
\draw [ultra thick][->] (0,0) -- (-.1,.4743);
\end{tikzpicture}
\end{tabular}
\caption{Probability of each move based on a spinner model.} \label{fig:spinner}
\end{figure}
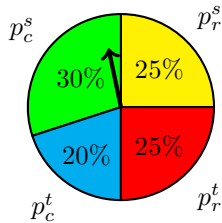
Similarly, $X_{\tim}$ can be thought of some means (such as a second spinner) of making the random decision that may be involved in the different types of moves.
%We refer to $p_c^t$ and $p_r^t$  as the \emph{tipsiness parameter} of the cop and robber, respectively.  
%The family of functions $c_{\tim}\colon V(G) \times V(G) \to V(G)$ is called a \emph{cop strategy}; $c_{\tim}(u,v) = w$ denotes the fact that if at time $\tim$ the cop is at $u$, the robber is at $v$, and the cop gets to decide where to move, then the cop will move to $w$. 
%Similarly, the family of functions $r_{\tim}\colon V(G) \times V(G) \to V(G)$ is called a \emph{robber strategy}. 
The functions $S_c$ and $S_r$ are called the \emph{cop strategy} and the \emph{robber strategy}, respectively.
We refer to moves in which the cop or robber get to employ their respective strategies as \emph{sober} cop or robber moves, and to moves where a random neighbor is selected as \emph{tipsy} cop or robber moves. We note that once both players have decided on a strategy and the initial position is fixed, the outcome of the game will only depend on the `spinner variables' $M_{\tim}$ and $X_{\tim}$.

We call $(\C_{\tim},\R_{\tim})$ the \emph{game state} at time $\tim$, and refer to $\C_{\tim}$ and $\R_{\tim}$ as the positions of the playing pieces. We say that the cop wins the game, if at some point the positions of the playing pieces coincide, the robber wins the game if this never happens. 
We say that the cop strategy $S_c$ is \emph{winning} against the robber strategy $S_r$ if there almost surely is a time step $\tim$ such that $\C_{\tim} = \R_{\tim}$ in the random process defined above. Conversely, we say that a robber strategy $S_r$ is winning against a cop strategy $S_c$ if $S_c$ is not winning against $S_r$. We note that in general this may depend on the distribution of starting positions.  However, if the graph we play on has at least one vertex of degree $\geq 3$, then from every starting position $(c_0,r_0)$ we can reach any possible game state $(c,r)$ by some finite sequence of moves before the cop wins the game (unless $c_0 = r_0$ or $c=r$). Hence, if the robber has a positive probability of winning the game from some starting position $(c_0,r_0)$, then they have a positive probability of winning from any starting position $(c_0',r_0')$, since there is a positive  probability that the game state will reach $(c_0,r_0)$ before either player makes a sober move or the cop catches the robber. In particular, in what follows we will always assume that the starting position is fixed, that is, there are vertices $c_0\neq r_0$ such that $\Prob[(\C_0, \R_0) = (c_0,r_0)] = 1$.

%Note that if the Markov chain $(\C_{\tim}, \R_{\tim})$ is recurrent, then the cop strategy $c_{\tim}$ is winning: in this case the probability of reaching any state in finite time is $1$, and in particular we will almost surely reach any state with $\C_{\tim}= \R_{\tim}$ in finite time.

We call a cop strategy $S_c$ superior to a cop strategy $S'_c$ against a robber strategy $S_r$ if  the corresponding random processes $(\C_{\tim},\R_{\tim})$ and $(\C'_{\tim},\R'_{\tim})$ satisfy
\[
    \forall (c,r) \in V\times V\colon \forall N\in \mathbb N  \colon \Prob_{(c,r)}[\exists\, \tim < N\colon \C_{\tim} = \R_{\tim}] \geq
    \Prob_{(c,r)}[\exists\, \tim < N\colon \C'_{\tim} = \R'_{\tim}].
\]

A cop strategy is called best possible in a class $\mathcal S$ of strategies against a given robber strategy $S_r$, if it is superior to any other strategy in $\mathcal S$ against $S_r$. 
Superior and best possible robber strategies can be defined analogously (with the converse inequality). 
We note that if $S_c$ is superior to a strategy $S'_c$ against $S_r$ and $S_c'$ is winning against $S_r$, then so is $S_c$. 
Thus there is no point in choosing $S'_c$ over $S_c$ in case the robber plays~$S_r$.

\section{Optimal and less than optimal strategies on the infinite grid}\label{section:grid_strategy}

In this section we describe best possible strategies for each player when playing on the infinite grid graph $\mathbb Z^2$. We also analyze for which values of the parameters $p_c^s$, $p_c^t$, $p_r^s$, and $p_r^t$ each of these strategies is winning.

Following P\'olya's approach, we start by simplifying the random process $(\C_{\tim},\R_{\tim})$ defined in the last section: instead of keeping track of both random walks, it suffices to keep track of the difference $\D_{\tim}:=\R_{\tim}-\C_{\tim}$.
In other words, we always think of the cop's position as being translated to the origin $(0,0)$, and  the current game state is uniquely determined by the robber's position $(x,y)$ in $\mathbb{Z}^2$.
 
We note that conditional on $\D_{\tim-1}$, tipsy cop and robber moves give the same probability distribution for $\D_{\tim}$.
Hence, to further simplify our notation, we combine $p_r^t$ and $p_c^t$ into one parameter which we call $p_t$, and drop the $s$ in $p_c^s$ and $p_r^s$
We now have the following parameters:
\begin{itemize}
\item $p_c = p_c^s$ is the probability of a sober move by the cop,
\item $p_r = p_r^s$  is the probability of a sober move by the robber,
\item $p_t= p_r^t+p_c^t$ is the probability of a tipsy move by either player,
\item $\D_{\tim}\in \mathbb Z^2$ is the difference between the cop's and robber's positions.
\end{itemize}

The optimal strategies and some other strategies we consider satisfy certain symmetry assumptions.
Firstly, they will only depend on  $\D_{\tim}$, and not on the specific positions $\C_{\tim}$ and $\R_{\tim}$. We adapt our definition of strategy functions accordingly, and let $\bar S_c(x,y)(x',y')$, and $\bar S_r(x,y)(x',y')$ denote the corresponding probability distributions.
Secondly, they will be (to some extent) symmetric under certain reflections of the grid; more precisely:
\[
    \bar S_c(x,y)(x',y') =
    \begin{cases}
        \bar S_c(-x,y)(-x',y') & \text{for } 0 < x \leq y,\\
        \bar S_c(y,x)(y',x') & \text{for } -x \leq y < x,\\
        \bar S_c(-x,-y)(-x',-y') & \text{for } y < -x.\\
    \end{cases}
\]
Note that under these symmetry conditions it is sufficient to define the strategies $\bar S_c$ and $\bar S_r$ for $y \geq x \geq 0$. We will consider the following strategies which (by the above symmetry assumptions) we only define for $y \geq x \geq 0$.
\begin{itemize}
    \item Robber strategy $\bar S_r^\uparrow$ always increases the $y$-coordinate, that is  $\bar S_r^\uparrow(x,y)(x,y+1) =1$.
    \item Cop strategy $\bar S_c^\downarrow$ always decreases the $y$-coordinate, that is  $\bar S_c^\downarrow(x,y)(x,y-1)=1$. %Note that if $x=y$ then by symmetry this is equivalent to decreasing the $x$-coordinate.
    \item Cop strategy $\bar S_c^{\downarrow} (p)$ always decreases the $y$-coordinate, unless $x=y$ in which case $\bar S_c^{\downarrow} (p)$ decreases the $y$-coordinate with probability $p$, and increases it with probability $1-p$. We note that $\bar S_c^{\downarrow} = \bar S_c^{\downarrow} (1)$.
    \item Cop strategy $\bar S_c^{\leftarrow}$ decreases the $x$-coordinate, unless the $x$-coordinate is 0, in which case $\bar S_c^{\leftarrow}$ decreases the $y$-coordinate.
\end{itemize}

%TO DO: for x=0, y=0, x=y etc. this imposes some restrictions on what is allowed (definition of strategies below do not adhere to these restrictions for x=y)

Throughout the rest of this section, we will let $\tilde \D_{\tim}$ denote an image of $\D_{\tim}$ in the quadrant $\{(x,y)\in \mathbb Z^2 \mid y \geq |x|\}$ under an appropriate sequence of reflections. More precisely, $\tilde \D_{\tim} = \phi(\D_{\tim} )$, where
\[
\phi(x,y) =
\begin{cases}
    (x,y) & \text{if } y \geq |x|,\\
    (y,x) & \text{if } x > |y|,\\
    (-y,-x) & \text{if } x < -|y|,\\
    (-x,-y) & \text{if } y \leq -|x|.
\end{cases}
\]
The reason we do not restrict $\tilde \D_{\tim}$ to the octant $\{(x,y)\in \mathbb Z^2 \mid y \geq x \geq 0\}$ is that in some proofs it will be easier to have only diagonal boundaries for the region of $\mathbb Z^2$ in which  $\tilde \D_{\tim}$ takes its values. We note that if the strategies of both players satisfy the above symmetry assumptions, then $\tilde \D_{\tim}$ is a Markov chain, but if they don't then this is not necessarily the case.
%
%Intuitively, it is clear that the cop should try to decrease the distance between the playing pieces while the robber should try to increase this distance. 
%Note that generally (unless $x=0$ or $y=0$) there will be two different ways of increasing the distance and two different ways of decreasing it. Thus there are several somewhat sensible strategies for both the cop and the robber. 
%The main result of this section shows that not all of these strategies are equally good.
%
%As mentioned above, throughout the game, the game state $\D_{\tim}$ can be described by a random walk on $\mathbb Z^2$, where the transition probabilities depend on the strategies employed by the cop and the robber, as well as on the probabilities $c$ and $r$ of a sober move by either player. 
%We note that this random walk is irreducible and time homogeneous.
We note that if $\tilde \D_{\tim}$ is a Markov chain, and this Markov chain is recurrent, then the cop almost surely wins; if it is positive recurrent then the expected time until this happens is finite. 
Conversely, if this Markov chain is transient, then the robber has a positive probability of winning.

Our main result in this section, Theorem \ref{thm:main on grid}, is as follows.
\begin{theorem}\label{thm:main on grid}
With parameters as defined above:
\begin{enumerate}
    \item If $p_r>p_c$, then $\bar S_r^\uparrow$ is winning against every cop strategy.
    \item If $p_c\geq p_r$, then $\bar S_c^\downarrow$ is winning against every robber strategy. If $p_c>p_r$, then the expected time until the game ends is finite.
    \item There is $\epsilon > 0$ depending on $p_t$ such that if $p_r < p_c < p_r + \epsilon$, then the cop strategy $\bar S_c^\leftarrow$ is not winning against $\bar S_r^\uparrow$.
\end{enumerate}
\end{theorem}

We remark that while $\bar S_r^\uparrow$ and $\bar S_c^\downarrow$ satisfy the above symmetry conditions, we do not assume that other strategies do. In other words, the first part of the above theorem implies that if $p_r>p_c$, then $\bar S_r^\uparrow$ is also winning against cop strategies which are not necessarily symmetric. The second part implies that if $p_c\geq p_r$, then $\bar S_c^\downarrow$ is also winning against robber strategies which do not satisfy the symmetry assumptions used in the definitions of $\bar S_r^\uparrow$ and $\bar S_c^\downarrow$.

The first part of Theorem~\ref{thm:main on grid} follows from Lemma \ref{lem:robber_wins} below. The second and third part will follow from results in Sections \ref{subsection:smart cop} and \ref{subsection:foolish cop}, respectively.

\begin{lemma}\label{lem:robber_wins} 
If $p_r>p_c$, then any robber strategy such that the distance between the position of the cop and the robber increases on every sober robber move is winning against any cop strategy. In particular, $\bar S_r^\uparrow$ is winning against any cop strategy.
\end{lemma}
\begin{proof}
Recall that we consider $\C_{\tim}$ and $\R_{\tim}$ as functions of i.i.d.\ processes $M_{\tim},X_{\tim}$ and that the construction depends on an enumeration of neighbors of $u$ and $v$ for every pair $(u,v)$ of vertices. Assume that this enumeration is so that $u_1$ and $u_2$ lie further away from $v$ than $u$, and $v_1$ and $v_2$ lie further away from $u$ than $v$. 

We define an auxiliary random process by
\[
    Y_{\tim} =
    \begin{cases}
        1 &\text{if $M_{\tim} = \mathrm{rs} $, or $M_{\tim} \in \{\mathrm{ct} ,\mathrm{rt} \}$ and $X_{\tim} \leq \frac 12$,}\\
        -1 &\text{otherwise,}
    \end{cases}
\]
and note that if $Y_{\tim} = 1$, then $\|\D_{\tim}\|_1 = \|\D_{\tim-1}\|_1 + 1$ where $\|.\|_1$ denotes (as usual) the sum of absolute values of the coordinates. In particular, $\|\D_{0}\|_1 + \sum_{i=1}^{\tim}Y_{i}\leq \|\D_{\tim}\|_1$ for all $\tim$. 

The random variables $Y_{\tim}$ are i.i.d.\, and $\Prob[Y_{\tim} = 1] = p_r + p_t/2 > 1/2$. Hence the probability that $\sum_{i=1}^{\tim}Y_{i}$ is non-negative for all $\tim$ is positive. Since the starting positions $c_0$ and $r_0$  are assumed to be distinct we know that $\|\D_{0}\|_1 > 0$, which implies that the probability that $\|\D_{\tim}\|_1>0$ for all $\tim$ is positive, and hence the robber wins the game with positive probability.
%
%Since the robber always increases the distance between themselves and the cop, we know that the probability of increasing the distance between the cop and robber in any round is at least $p_r+p_t/2$ (it may be larger depending on the current game state and the cop's strategy), and consequently the probability of the distance decreasing is at most $p_c + p_t/2$.
%Since $p_r < p_c$, we conclude that $p_r+p_t/2 > 1/2$ and thus the robber wins with positive probability by Proposition \ref{prp:inhomogeneous_recurrent}.
\end{proof}

In the remainder of this section, we focus on the case $p_c \geq p_r$. 
As one might expect, in this case the cop has a winning strategy, and in case $p_c > p_r$ the cop can ensure that the game has finite expected length. 
This is shown in Subsection~\ref{subsection:smart cop}. In contrast to Lemma~\ref{lem:robber_wins}, this depends on the strategy chosen by the cop; even if $p_c>p_r$ there are cop strategies always decreasing the distance between the cop and the robber which do not almost surely lead to a cop win. 
One such strategy is $\bar S_c^\leftarrow$ as discussed in Subsection \ref{subsection:foolish cop}.

\subsection{Cop playing an optimal strategy}\label{subsection:smart cop}

In this section, we show that as long as $p_c \geq p_r$ the cop has a winning strategy. The section is organized as follows: we start by showing in Lemma~\ref{lemma:robber_strategy} that $\bar S_r^\uparrow$ is the best possible robber strategy against $\bar S_c^\downarrow$, in particular, if   $S_r^\uparrow$ is not winning against $\bar S_c^\downarrow$, then no robber strategy is.
After that, we show that $\bar S_c^\downarrow$ is almost surely winning against $\bar S_r^\uparrow$ given that $p_c \geq p_r$. 
The proof of this result is divided into two lemmas: we begin by showing that $\bar S_c^\downarrow$ is the best possible strategy against $\bar S_r^\uparrow$ in the family $\mathcal S := \{\bar S_c^\downarrow(p)\mid 0\leq p \leq 1\}$, and then we give a strategy in the family $\mathcal S$ which is winning against $\bar S_r^\uparrow$.

%Throughout the rest of this section we use the following model to determine tipsy moves in the game. Let $\tilde \D_{\tim}$ be defined as above.  Let $X_\tim$ be chosen uniformly at random from the set of canonical unit vectors $\{(1,0),(-1,0),(0,1),(0,-1)\}$, and pick $\D_{\tim+1} = \D_{\tim} + \varphi^{-1} (X_{\tim})$, where $\varphi$ is an automorphisms of $\mathbb Z^2$ fixing the origin $(0,0)$, and mapping $\D_{\tim}$ to $\tilde \D_{\tim}$. In order to avoid ambiguity, we arbitrarily fix one such map $\varphi$ for each vertex where the map is not unique (for instance, we can let $\varphi$ be the branch of the function $\phi$ which maps $\D_{\tim}$ to $\tilde \D_{\tim}$).

%Let $S_c$ and $S_c'$ be cop strategies, let $S_r$  and $S_r'$ be robber strategies, and denote by $\D_{\tim}$ and $\D_{\tim}'$ the respective random sequences of game states. We define the \emph{standard coupling} of $\D_{\tim}$ and $\D_{\tim}'$ as the coupling where both games use the same sequence $M_{\tim}$ of spinner outcomes (that is, if a sober cop or robber move happens in the one game, then also in the other), and both games use the same sequence of random directions $X_{\tim}$ in their respective tipsy moves.

\begin{lemma} \label{lemma:robber_strategy}
The robber strategy $\bar S_r^\uparrow$ is superior to any other robber strategy $S_r$ against $\bar S_c^\downarrow$.
%In particular, if the cop plays $\bar S_c^\downarrow$, then the robber has a symmetric strategy which is winning with positive probability if and only if $\bar S_r^\uparrow$ is winning with positive probability.
\end{lemma}

\begin{proof}
Consider the standard coupling of $\D_{\tim}$ and $\D_{\tim}'$, where $\D_{\tim}$ is the sequence of game states with respect to the strategies $\bar S_c^\downarrow$ and $\bar S_r^\uparrow$, and $\D_{\tim}'$ is the sequence of game states with respect to the strategies $\bar S_c^\downarrow$ and $S_r$, for an arbitrary robber strategy $S_r$. Define $\tilde \D_{\tim}$ and $\tilde \D_{\tim}'$ as before. Note that by definition neither player is allowed to pass in either of the two games. Hence $\|\D_{\tim}- \D_{\tim}'\|_1$ and thus also $\|\tilde \D_{\tim}-\tilde \D_{\tim}'\|_1$ is even at all times.

To ensure that the effects of tipsy moves in the two games are compatible, we define orderings of the neighbors of all vertices (slightly different to those used in the proof of Lemma \ref{lem:robber_wins}).
For $v-u \in \{(x,y) \in \mathbb R^2 \mid y \geq |x|\}$ define $u_1, \dots, u_4$ by
$u-u_1 = (0,-1)$, $u-u_2 = (-1,0)$, $u-u_3 = (1,0)$, and $u-u_4 = (0,1)$. Enumerations for the remaining vertices are determined by diagonal symmetries so that tipsy cop moves have the same effect on $\tilde \D_{\tim}$, no matter whether $\tilde \D_{\tim} = \D_{\tim}$, or $\tilde \D_{\tim} \neq \D_{\tim}$.
Moreover, we define $v_1, \dots, v_4$ by $u-u_i = v_i-v$ (so that the effect of tipsy cop and robber moves on $\D_{\tim}$ only depends on $X_{\tim}$)

We prove by induction on $\tim$ that if $\tilde \D_{\timtwo} ' \neq (0,0)$ for all $\timtwo \leq \tim$, then $\tilde \D_{\tim}-\tilde \D_{\tim}' \in \{(a,b) \in \mathbb Z^2 \mid b \geq |a|\}$. In particular, the robber cannot lose the game playing strategy $S_r^{\uparrow}$ without first losing the coupled game where they play strategy $S_r$. For the base case, we note that this is true for $\tim =0$ by the fact that starting positions in the standard coupling are the same. 

For the induction step, let $\tilde \D_{\tim} = (x,y)$ and $\tilde \D_{\tim}' = (x',y')$ and assume that $y-y' \geq |x'-x|$. 
We first note that if $y-y' > |x'-x|$, then  $y-y' \geq |x'-x| + 2$ because $\|\tilde \D_{\tim}-\tilde \D_{\tim}'\|_1$ is even. Since $\tilde \D_{\tim}$ and $\tilde \D_{\tim+1}$ differ in at most one coordinate and the difference is at most 1, and the same is true for $\tilde \D_{\tim}'$ and $\tilde \D_{\tim+1}'$, this implies that in this case the inductive claim remains valid at step $\tim+1$.

We may thus assume without loss of generality that  $y-y' = |x'-x|$. We distinguish cases according to the different types of moves that can occur at time $\tim$.

If we have a sober robber move at time $\tim$, then the $y$-coordinate of $\tilde \D_{\tim+1}$ is larger than that of $\tilde \D_{\tim}$, and thus $\tilde \D_{\tim+1}-\tilde \D_{\tim}' \in \{(a,b) \mid b \geq|a|+1\}$. Since $\tilde \D_{\tim}'$ and $\tilde \D_{\tim+1}'$ differ in at most one coordinate and the difference is at most 1, this implies that the inductive claim remains valid at step $\tim+1$.

Next consider the case when we have a sober cop move. If $y' \neq |x'|$, then $\D_{\tim +1}' = (x',y'-1)$, and thus $\tilde \D_{\tim+1}-\tilde \D_{\tim}' \in \{(a,b) \mid b \geq |a|+1\}$. The same argument as above shows that the inductive claim remains valid at step $\tim+1$. If $y' = x'$ then either $y=x$, or $y = x' + y' - x = 2 y' -x$ due to our assumption $y-y' = |x'-x|$. In the first case $\tilde \D_{\tim+1}' = (x'-1,y')$ and $\tilde \D_{\tim+1} = (x-1,y)$, and the inductive claim remains valid at step $\tim+1$ since $(y-y') = |x'-x| = |(x'-1)-(x-1)|$. In the second case $\tilde \D_{\tim+1}' = (x'-1,y')$ and $\tilde \D_{\tim+1} = (x,y-1)$, and the inductive claim remains valid at step $\tim+1$ since $(y-1-y') = y'-1-x = x'-1-x$.
The case $y=-x$ follows by a symmetric argument.

Finally consider the case when we have a tipsy move by either player. We further distinguish cases based on $X_s$. If $X_s \in [0,\frac 14)$, then the exact same argument as in the case of sober robber moves applies. If $X_s \in [\frac 34,1)$, then the exact same argument as in the case of sober cop moves applies. If $X_s \in [\frac 14, \frac 12)$ and $x' \neq y'$, then $x \neq y$, and thus $\tilde \D_{\tim+1}' = (x'+1,y')$ and $\tilde \D_{\tim+1} = (x+1,y)$ and it is easy to check that the inductive claim remains valid. If $x' = y'$, then either $y=x$, or $y = x' + y' - x = 2 y' -x$ as above.
In the first case $\tilde \D_{\tim+1}' = (x',y'+1)$ and $\tilde \D_{\tim+1} = (x,y+1)$, and the inductive claim remains valid at step $\tim+1$ since $(y+1)-(y'+1) = y-y' |x'-x|$. In the second case $\tilde \D_{\tim+1}' = (x',y'+1)$ and $\tilde \D_{\tim+1} = (x+1,y)$, and the inductive claim remains valid at step $\tim+1$ since $(y-y'-1) = y'-x-1 = x'-(x+1)$. The case $X_s \in [\frac 12, \frac 34)$ again follows by a symmetric argument.

This finishes the inductive step and thus also the proof.
\end{proof}

\begin{lemma}\label{lemma:cop_strategy}
The cop strategy $S_c^\uparrow$ is superior to any strategy of type $S_c^\uparrow(p)$ against $S_r^\downarrow$.
\end{lemma}

\begin{proof} 

As in the proof of Lemma \ref{lemma:robber_strategy}, consider the standard coupling of $\D_{\tim}$ and $\D_{\tim}'$, where $\D_{\tim}$ is the sequence of game states with respect to the strategies $\bar S_c^\downarrow$ and $\bar S_r^\uparrow$, and $\D_{\tim}'$ is the sequence of game states with respect to the strategies $\bar S_c^\downarrow(p)$ and $\bar S_r^\uparrow$. 
Define $\tilde \D_{\tim}$ and $\tilde \D_{\tim}'$ as above. An inductive argument based on a case-by-case analysis like in the proof of Lemma \ref{lemma:robber_strategy} shows that if $\tilde \D_{\timtwo} ' \neq (0,0)$ for all $\timtwo \leq \tim$, then $\tilde \D_{\tim}' - \tilde \D_{\tim} \in \{(a,b) \in \mathbb Z^2 \mid b \geq |a|\}$. 
We leave the details to the reader.
\end{proof}

\begin{lemma} \label{lemma:positive_recurrent} 
Assume that $p_r < p_c$. For $p = \frac{p_t(p_c-p_r)}{2p_c(2p_r+p_t)}$, the cop strategy $S_c^\downarrow(p)$ is winning against the robber strategy $S_r^\uparrow$, and the expected duration of the game is finite.
\end{lemma}

\begin{proof}
We first note that $0 \leq p \leq 1$, so the strategy $S_c^\downarrow(p)$ is well defined. Let $\D_{\tim}$ be the sequence of game states with respect to strategies $S_c^\downarrow(p)$ and $S_r^\uparrow$, and let $\tilde \D_{\tim}$ be defined as above. Due to our symmetry assumptions, $\tilde \D_{\tim}$ is a random walk on the induced subgraph of the integer grid with vertex set $\{(x,y) \in \mathbb Z^2 \mid y \geq x \geq 0\}$. 

We start by computing transition probabilities of this random walk. The only possible transition in this graph starting at $(0,0)$ is to $(0,1)$, so $p_{(0,0),(0,1)}=1$.
For $(x,y) \neq (0,0)$ the transition probabilities are given by
\[
    p_{(x,y),(x',y')} =
    \begin{cases}
        p_r+ p_t/4& \text{if } |x| \neq y \text{ and } (x',y') = (x,y+1)\\
        p_c+p_t/4 & \text{if } |x| \neq y \text{ and } (x',y') = (x,y-1)\\
         p_t/2 &\text{if } |x| \neq y \text{ and } (x',y') = (x\pm 1,y)\\
        p \cdot p_c+ p_t/2 &\text{if } x = \pm y \text{ and } (x',y') = (x \mp 1,y)\\
        (1-p) p_c+p_r+ p_t/2 &\text{if } |x| = y \text{ and } (x',y') = (x,y+1).
    \end{cases}
\]

We claim that these are precisely the transition probabilities of a random walk given by edge weights, that is, there is a function $w \colon E(G) \to \mathbb R_+$ such that
\[p_{ab} = \frac{w(ab)}{\sum_{b' \sim a} w(ab')}.\] 
We will show that this function is given by
\begin{align*}
    w((x,y),(x,y+1)) &= \left( \frac{p_r+p_t/4}{p_c+p_t/4}\right)^y \text{ and}\\
    w((x,y),(x+1,y)) &= \frac{p_t/4}{p_c+p_t/4}\left( \frac{p_r+p_t/4}{p_c+p_t/4}\right)^{y-1}.
\end{align*}

To prove this claim, first consider a vertex $a =(x,y)$ with $y \neq |x|$. Then 
\[
\sum_{b' \sim a} w(ab') = \left( \frac{p_r+p_t/4}{p_c+p_t/4}\right)^y + \left( \frac{p_r+p_t/4}{p_c+p_t/4}\right)^{y-1} + 2 \cdot \frac{p_t/4}{p_c+p_t/4}\left( \frac{p_r+p_t/4}{p_c+p_t/4}\right)^{y-1} = \frac{(p_r+p_t/4)^{y-1}}{(p_c+p_t/4)^y}
\]
and it is easy to verify that $p_{ab} = {w(ab)}/{\sum_{b' \sim a} w(ab')}$ for every neighbor $b$ of $a$. Next consider a vertex $a = (x,x)$; the case $a = (x,-x)$ is analogous. We only have two incident edges and thus obtain
%\[\sum_{b' \sim a} w(ab') = \left( \frac{r+t/4}{c+t/4}\right)^y +  \frac{t/4}{c+t/4}\left( \frac{r+t/4}{c+t/4}\right)^{y-1} = \frac{(r+t/4)^{y-1}}{(c+t/4)^y} \cdot \left(r + \frac t2\right). \]
\[
\sum_{b' \sim a} w(ab') = \left( \frac{p_r+p_t/4}{p_c+p_t/4}\right)^y +  \frac{p_t/4}{p_c+p_t/4}\left( \frac{p_r+p_t/4}{p_c+p_t/4}\right)^{y-1} = \frac{(p_r+p_t/4)^{y-1}}{(p_c+p_t/4)^y} \cdot \left(p_r + \frac{p_t}{2}\right).
\]

Letting $b = (x-1,y)$, we obtain
\begin{align*}
    \frac{w(ab)}{\sum_{b' \sim a} w(ab')} 
    = \frac{p_t/4}{p_r+p_t/2}
    %= \frac{p_t (p_c+p_r+p_t)}{2(2 p_r+p_t)}
    %= \frac{p_t (p_c-p_r+2p_r+p_t)}{2(2 p_r+p_t)}
    = \frac{p_t (p_c-p_r)}{2(2 p_r+p_t)} + \frac{p_t (2p_r+p_t)}{2(2 p_r+p_t)}
    = p\cdot p_c+p_t/2
\end{align*}
as claimed. Since $a$ has only two neighbors, and transition probabilities must add up to $1$, this concludes the proof of our claim.

By Theorem \ref{thm:rw-edgeweights}, this random walk is positive recurrent if and only if
$\sum_{a,b} w(a,b) < \infty$.
Hence it suffices to show that 
\[\sum_{y \geq 0}\left(\sum_{-y\leq x\leq y}w((x,y),(x,y+1))\right)+
\sum_{y \geq 0}\left(\sum_{-y\leq x<y}w((x,y),(x+1,y))\right)
<\infty.\]
Letting $\Ratb=\frac{p_r+p_t/4}{p_c+p_t/4}$ and $\Rata = \frac{p_t/4}{p_c+p_t/4}$ and noting that $\Ratb < 1$, we obtain
\begin{align*}
    \sum_{y \geq 0}\left(\sum_{-y\leq x\leq y}w((x,y),(x,y+1))\right)
    =
    \sum_{y\geq 0}\left(\sum_{-y \leq x \leq y} \Ratb 
    ^y\right)= \sum_{y \geq 0} (2y+1) \Ratb^y =\frac{\Ratb+1}{(\Ratb-1)^2} <\infty,
    % &=2\sum_{y\geq 0}\sum_{0\leq x \leq y} \Ratb 
    % ^y- \sum_{y\geq 0}\Ratb^y=
\end{align*}
and 
\begin{align*}
\sum_{y \geq 0}\left(\sum_{-y\leq x< y}w((x,y),(x+1,y))\right) =\sum_{y\geq 1}\left(\sum_{-y\leq x<y}\Rata\Ratb^{y-1} \right) =\sum_{y\geq 1}2y\Rata \Ratb^{y-1}=\frac{2\Rata}{(\Ratb-1)^2}<\infty.
\end{align*}
Thus, $\tilde \D_{\tim}$ is positive recurrent and in particular the expected number of steps until $\tilde \D_{\tim}=(0,0)$ is finite; in other words, the cop wins the game almost surely and the expected time until this happens is finite.
\end{proof}

Finally, let us consider the case $p_r=p_c$; note that Lemmas \ref{lemma:robber_strategy} and  \ref{lemma:cop_strategy} did not assume anything about the value of $p_c$ and $p_r$, so they still apply in this case.

\begin{lemma}
If $0 \leq p_r=p_c \leq 1/2$, then $S_c^\downarrow$ is winning against $S_r^\uparrow$, but the expected length of the game is infinite.
\end{lemma}
\begin{proof}
If $p_t = 0$, then $p_r=p_c=1/2$, and thus $\D_{\tim}$ is a simple random walk on the $y$-axis of $\mathbb{Z}^2$ which is null-recurrent.  

If $p_t > 0$, let $\tilde \D_{\tim}'$ be the sequence of game states with respect to the strategies $S_c^\downarrow(0)$  and $S_r^\uparrow$. The same computations as in Lemma \ref{lemma:positive_recurrent} show that $\tilde \D_{\tim}'$ is a random walk defined by edge weights 
\[
    w((x,y),(x+1,y))=\left( \dfrac{p_t/4}{p_c+p_t/4}\right)\quad\mbox{and}\quad w((x,y),(x,y+1)) =1.
\]
Note that there are only two different edge weights both of which are non-zero. Hence, if $\tilde \D_{\tim}'$ was transient, then the random walk with all edge weights equal to $1$ would also be transient by Theorem~\ref{thm:RMP}, but we know that a simple random walk on $\mathbb Z^2$ and thus also the simple random walk on any subgraph of $\mathbb Z^2$ is recurrent. 
So $S_c^\downarrow(0)$  is winning against $S_r^\uparrow$. By Lemmas~\ref{lemma:robber_strategy}, and~\ref{lemma:cop_strategy}, $S_c^\downarrow$ is a winning strategy for the cop against any robber strategy.

To show that the expected duration of the game is infinite, let $\D_{\tim}$ be the sequence of game states and let $\tilde \D_{\tim}$ be defined as above. Consider the random process $Y_{\tim}$ on the integers with initial position $Y_0 =\|\D_{0}\|_1$ and the following transition rules. 
If $\tilde \D_{\tim} = (0,y)$ for some $y$, $M_{\tim} \in \{\mathrm{ct} ,\mathrm{rt} \}$, and $\tilde \D_{\tim+1} - \tilde \D_{\tim} = (0,1)$ we let $Y_{\tim+1} = Y_{\tim}-1$. In all other cases, we let $Y_{\tim+1} = Y_{\tim} + \|\tilde \D_{\tim+1}\|_1 - \|\tilde \D_{\tim}\|_1$. 

Note that by choosing appropriate enumerations of neighbours, we can ensure that the increments $Y_{\tim+1} - Y_{\tim}$ only depend on $M_\tim$ and $X_\tim$ and hence they are i.i.d\ uniform on $\{1,-1\}$. It follows that $Y_{\tim}$ is a simple random walk on the integers. Further note that $Y_{\tim} \leq \|\tilde \D_{\tim}\|_1$. Thus, if the expected duration of the game was finite, then the expected time until $Y_{\tim}$ reaches $0$ would be finite as well, which contradicts Theorem \ref{thm:gamblersruin}.
\end{proof}

\subsection{Cop playing foolish strategy}\label{subsection:foolish cop}
The aim of this subsection is to show that the cop should not deviate too far from strategy $S_c^\downarrow$ if they want to maximize their chances of winning.
More precisely, if the cop plays the strategy $S_c^\leftarrow$, and the robber plays strategy $S_r^\uparrow$, then the robber has positive probability of winning for some values $p_c > p_r$.

\begin{theorem}
If $p_c < \left(1+\frac{p_t}{4+p_t}\right) p_r$, then $S_r^\uparrow$ is winning against $S_c^\leftarrow$.
\end{theorem}
\begin{proof}
As in previous proofs, we start by defining an enumeration of the neighbors of $u$ and $v$ with respect to the pair $(u,v)$. As usual, write $\|(x,y)\|_{\infty}$ for $\max (|x|,|y|)$.  Let $u_1$ be a neighbor of $u$ minimising $\|v-u_i\|_{\infty}$, let $u_4$ be a neighbor of $u$ maximising $\|v-u_i\|_{\infty}$, let $u_2$ be a neighbor such that $\|v-u_i\|_{\infty} =\|v-u\|_{\infty}$, and let $u_3$ be the remaining neighbor. Note that choosing such an enumeration is always possible, but it is not necessarily unique. In order to ensure that random moves by either player have the same effect in the game, let $v_i$ be the neighbor of $v$ satisfying $u-u_i = v_i - v$.

Recursively define random variables $Y_n$ and $t_n$ depending on the sequences $M_{\tim}$ and $X_{\tim}$ as follows. Let $t_0 = 0$ and $Y_0 = 1$. For $n \geq 0$, we distinguish the following cases:
\begin{itemize}
    \item If $M_{t_n} = \mathrm{rs} $, or $M_{t_n} \in \{\mathrm{ct} , \mathrm{rt} \}$ and $X_{t_n} < \frac 14$, then we set $Y_{n+1}= Y_{n}+1$ and $t_{n+1} = t_n+1$.
    \item If $M_{t_n} = \mathrm{cs} $, or $M_{t_n} \in \{\mathrm{ct} , \mathrm{rt} \}$ and $X_{t_n} \geq \frac 34$, then we set $Y_{n+1}= Y_{n}-1$ and $t_{n+1} = t_n+1$.
    \item If neither of the above cases applies, then we set $t_{n+1} = t_n+2$ and further distinguish cases for $Y_{n+1}$.
    \begin{itemize}
        \item If $M_{t_{n}+1} = \mathrm{rs} $, or $M_{t_{n}+1} \in \{\mathrm{ct} , \mathrm{rt} \}$ and $X_{t_{n}+1} < \frac 14$, then we set $Y_{n+1}= Y_{n}+1$.
        \item If $\frac 14 \leq X_{t_{n}} < \frac 12$ and $M_{t_{n}+1} = \mathrm{rs} $, or $M_{t_{n}+1} \in \{\mathrm{ct} , \mathrm{rt} \}$ and $\frac 14 \leq X_{t_{n+1}} < \frac 12$, then we set $Y_{n+1}= Y_{n}-1$.
        \item Otherwise, we set $Y_{n+1}= Y_{n}$.
    \end{itemize} 
\end{itemize}

We note that $\|\D_{t_{n+1}}-\D_{t_n}\|_{\infty} \leq 1$ due to the definition of $t_n$ and the ordering of the neighbors defined above. Moreover, if $\|\D_{t_{n+1}}\|_{\infty}-\|\D_{t_n}\|_{\infty} = -1$, then also $Y_{n+1} - Y_n = -1$. Hence $Y_n \leq \|\D_{t_n}\|_{\infty}$. We further note that $Y_n$ is in fact a random walk on $\mathbb Z$ with transition probabilities
\begin{align*}
    p_{n,n+1} = \Prob[Y_{n+1} = y+1\mid Y_{n} = y] &= p_r + \frac {p_t}4 + \frac {p_t}2 \left(p_r+\frac {p_t}4 \right),\\ 
    p_{n,n-1} =\Prob[Y_{n+1} = y-1\mid Y_{n} = y] &= p_c + \frac {p_t}4   + \frac {p_t}4 p_c + \frac {p_t}2  \cdot \frac {p_t}4 .
\end{align*}
Note that these do not add up to $1$ since $\Prob[Y_{n+1} = y\mid Y_{n} = y]$  is non-zero.

If $p_c < \left(1+\frac{p_t}{4+p_t}\right) p_r$, then $p_{n,n+1} > p_{n,n-1}$, and Theorem \ref{thm:gamblersruin} implies that $\Prob[\forall n > 0\colon Y_n >1] >0$. This implies that the probability that  $\|\D_{t_n}\|_{\infty} \geq 2$ for every $n>0$ is positive, and therefore the robber wins the game with positive probability.
\end{proof}

\section{Game on Trees}\label{section:tree_strategy}
We revert back to parameters $p_c^s$, $p_r^s$, $p_c^t$, and $p_r^t$. 
Throughout this section we further assume that $p_c^s+p_c^t = p_r^s + p_r^t= 1/2$, or in other words, no player gets to make significantly more moves than the other; the only imbalance in the game comes from the tipsiness parameters of the two players. In particular, the parameters $p_c^s$ and $p_r^s$ are determined by $p_c^t$ and $p_r^t$.

We start by considering infinite $\delta$-regular trees. Consider the cop strategy $S_c^\ast$ in which every sober cop move decreases the distance between the cop and the robber, and the robber strategy $S_r^\ast$ in which every sober robber move increases this distance. We note that there are many different ways for the robber to increase the distance, but they are all equivalent in the sense that any pair of vertices at distance $d$ in a $\delta$-regular trees can be mapped to any other such pair of vertices by an automorphism of the tree.

If the cop plays strategy $S_c^\ast$, then the probability that the distance between the two players decreases in any given step is bounded below by $p_c^s + (p_c^t+p_r^t)\frac{1}{\delta}$.
If the robber plays strategy $S_r^\ast$, then the probability that the distance between the two players increases in any given step is bounded below by $p_r^s + (p_c^t+p_r^t)\frac{\delta-1}{\delta}$.
Substituting $p_c^s = 1/2-p_c^t$ and $p_r^s=1/2-p_r^t$ and applying Proposition \ref{prp:inhomogeneous_recurrent} we arrive at the following conclusion.
\begin{lemma} \label{lemma:delta}
Assume that $p_c^s+p_c^t = p_r^s + p_r^t= 1/2$ for a game on a $\delta$-regular tree.
If $p_r^t \geq p_c^t (\delta-1)$, then $S_c^\ast$ is winning against every robber strategy. Otherwise $S_r^\ast$ is winning against every cop strategy. 
\end{lemma}

 \begin{comment}
For instance, if $p_r^t=1/2$ so $r=0$ and $c+p_c^t=1/2$ then
 $ (1/2-c+1/2)\frac{\delta-1}{\delta}$ is the probability that they move further away in a given round and 
 $c + (1/2-c+1/2)\frac{1}{\delta}$ is the probability that they move closer. 
 Setting these probabilities equal we have
 $$ (1/2-c+1/2)\frac{\delta-1}{\delta}= c + (1/2-c+1/2)\frac{1}{\delta}.$$
 Solving for $c$ we
find that $c=\frac{\delta-2}{2(\delta-1)}$, and $$\lim_{\delta \to \infty} c =\lim_{\delta \to \infty} \frac{\delta-2}{2(\delta-1)}  =\lim_{\delta \to \infty} \frac{1-\frac{2}{\delta} }{2(1-\frac{1}{\delta})} = \frac{1}{2}  $$
\end{comment} 
 \begin{comment}
 Solving for $c$ we find that
 \begin{align*}
  (1-c)\frac{\delta-1}{\delta} & = c + (1-c)\frac{1}{\delta} \\
 (1-c)(\delta-1) & = (\delta -1) c + 1  \\
 (1-2c)& =  \frac{1}{\delta-1}  \\
 2c& = - \frac{1}{\delta-1} +  \frac{\delta-1}{\delta-1} \\
 c& =  \frac{\delta-2}{2(\delta-1)} \\
    \end{align*}\end{comment}

We now consider a more general class of trees which can be described as $\delta$-regular trees rooted to a $(\Delta-1)$-regular tree.
To be precise, let $\delta > 1$, and let $T$ be the tree that is $\delta$-regular everywhere except the root, which has degree $\delta-1$.
Let $\Delta > \delta$, and let $B$ be a $(\Delta-1)$-regular tree. 
Then the tree $X(\Delta,\delta)$ is constructed by connecting the root of a copy of $T$ to each node in $B$, see Figure \ref{fig:examples} for examples. We call $B$ the \emph{base tree} of $X(\Delta,\delta)$ and we refer to copies of the $\delta$-regular tree in $X(\Delta,\delta)$ as  \emph{small trees}. Clearly, each node in the base of $X(\Delta,\delta)$ has degree $\Delta$, and each node in a small tree has degree $\delta$.

\begin{figure}
\tikzset{every picture/.style={line width=0.75pt}} %set default line width to 0.75pt        
\resizebox{.3\textwidth}{!}{
\begin{tikzpicture}[x=0.75pt,y=0.75pt,yscale=-1,xscale=1]
%uncomment if require: \path (0,404); %set diagram left start at 0, and has height of 404

%Shape: Circle [id:dp946601855754976] 
\draw  [fill={rgb, 255:red, 0; green, 0; blue, 0 }  ,fill opacity=1 ] (174,254.25) .. controls (174,251.9) and (175.9,250) .. (178.25,250) .. controls (180.6,250) and (182.5,251.9) .. (182.5,254.25) .. controls (182.5,256.6) and (180.6,258.5) .. (178.25,258.5) .. controls (175.9,258.5) and (174,256.6) .. (174,254.25) -- cycle ;
%Shape: Circle [id:dp7206168166683544] 
\draw  [fill={rgb, 255:red, 0; green, 0; blue, 0 }  ,fill opacity=1 ] (141.25,232.75) .. controls (141.25,230.4) and (143.15,228.5) .. (145.5,228.5) .. controls (147.85,228.5) and (149.75,230.4) .. (149.75,232.75) .. controls (149.75,235.1) and (147.85,237) .. (145.5,237) .. controls (143.15,237) and (141.25,235.1) .. (141.25,232.75) -- cycle ;
%Shape: Circle [id:dp649967295899101] 
\draw  [fill={rgb, 255:red, 0; green, 0; blue, 0 }  ,fill opacity=1 ] (199.08,233.41) .. controls (199.55,231.11) and (201.79,229.62) .. (204.09,230.08) .. controls (206.39,230.55) and (207.88,232.79) .. (207.42,235.09) .. controls (206.95,237.39) and (204.71,238.88) .. (202.41,238.42) .. controls (200.11,237.95) and (198.62,235.71) .. (199.08,233.41) -- cycle ;
%Shape: Ellipse [id:dp9015664616657586] 
\draw   (86,242.5) .. controls (86,261.55) and (130.88,277) .. (186.25,277) .. controls (241.62,277) and (286.5,261.55) .. (286.5,242.5) .. controls (286.5,223.45) and (241.62,208) .. (186.25,208) .. controls (130.88,208) and (86,223.45) .. (86,242.5) -- cycle ;
%Shape: Circle [id:dp10851698481081196] 
\draw  [color={rgb, 255:red, 208; green, 2; blue, 27 }  ,draw opacity=1 ][fill={rgb, 255:red, 208; green, 2; blue, 27 }  ,fill opacity=1 ] (141.36,131.25) .. controls (141.28,128.9) and (143.12,127) .. (145.47,127) .. controls (147.82,127) and (149.78,128.9) .. (149.86,131.25) .. controls (149.94,133.6) and (148.1,135.5) .. (145.75,135.5) .. controls (143.4,135.5) and (141.44,133.6) .. (141.36,131.25) -- cycle ;
%Shape: Circle [id:dp8339673252972051] 
\draw  [color={rgb, 255:red, 208; green, 2; blue, 27 }  ,draw opacity=1 ][fill={rgb, 255:red, 208; green, 2; blue, 27 }  ,fill opacity=1 ] (141.64,109.25) .. controls (141.56,106.9) and (143.4,105) .. (145.75,105) .. controls (148.1,105) and (150.06,106.9) .. (150.14,109.25) .. controls (150.22,111.6) and (148.38,113.5) .. (146.03,113.5) .. controls (143.68,113.5) and (141.72,111.6) .. (141.64,109.25) -- cycle ;
%Straight Lines [id:da1715927247888378] 
\draw [color={rgb, 255:red, 208; green, 2; blue, 27 }  ,draw opacity=1 ]   (145.89,109.25) -- (145.77,118.99) -- (145.61,131.25) ;
%Shape: Circle [id:dp5553048038499598] 
\draw  [color={rgb, 255:red, 208; green, 2; blue, 27 }  ,draw opacity=1 ][fill={rgb, 255:red, 208; green, 2; blue, 27 }  ,fill opacity=1 ] (141.36,178.25) .. controls (141.28,175.9) and (143.12,174) .. (145.47,174) .. controls (147.82,174) and (149.78,175.9) .. (149.86,178.25) .. controls (149.94,180.6) and (148.1,182.5) .. (145.75,182.5) .. controls (143.4,182.5) and (141.44,180.6) .. (141.36,178.25) -- cycle ;
%Shape: Circle [id:dp7301823364631214] 
\draw  [color={rgb, 255:red, 208; green, 2; blue, 27 }  ,draw opacity=1 ][fill={rgb, 255:red, 208; green, 2; blue, 27 }  ,fill opacity=1 ] (141.64,156.25) .. controls (141.56,153.9) and (143.4,152) .. (145.75,152) .. controls (148.1,152) and (150.06,153.9) .. (150.14,156.25) .. controls (150.22,158.6) and (148.38,160.5) .. (146.03,160.5) .. controls (143.68,160.5) and (141.72,158.6) .. (141.64,156.25) -- cycle ;
%Straight Lines [id:da4830927612071555] 
\draw [color={rgb, 255:red, 208; green, 2; blue, 27 }  ,draw opacity=1 ]   (145.89,156.25) -- (145.77,165.99) -- (145.61,178.25) ;
%Straight Lines [id:da12292526431844353] 
\draw [color={rgb, 255:red, 208; green, 2; blue, 27 }  ,draw opacity=1 ]   (146.17,134.25) -- (146.04,143.99) -- (145.89,156.25) ;
%Straight Lines [id:da5892022237242749] 
\draw [color={rgb, 255:red, 208; green, 2; blue, 27 }  ,draw opacity=1 ]   (145.75,177.5) -- (145.63,187.24) -- (145.5,228) ;
%Straight Lines [id:da351322192588851] 
\draw [color={rgb, 255:red, 208; green, 2; blue, 27 }  ,draw opacity=1 ]   (145.87,95.26) -- (145.75,105) -- (145.89,109.25) ;
%Shape: Circle [id:dp142852184943556] 
\draw  [color={rgb, 255:red, 208; green, 2; blue, 27 }  ,draw opacity=1 ][fill={rgb, 255:red, 208; green, 2; blue, 27 }  ,fill opacity=1 ] (200.36,135.25) .. controls (200.28,132.9) and (202.12,131) .. (204.47,131) .. controls (206.82,131) and (208.78,132.9) .. (208.86,135.25) .. controls (208.94,137.6) and (207.1,139.5) .. (204.75,139.5) .. controls (202.4,139.5) and (200.44,137.6) .. (200.36,135.25) -- cycle ;
%Shape: Circle [id:dp9678017593925394] 
\draw  [color={rgb, 255:red, 208; green, 2; blue, 27 }  ,draw opacity=1 ][fill={rgb, 255:red, 208; green, 2; blue, 27 }  ,fill opacity=1 ] (200.64,113.25) .. controls (200.56,110.9) and (202.4,109) .. (204.75,109) .. controls (207.1,109) and (209.06,110.9) .. (209.14,113.25) .. controls (209.22,115.6) and (207.38,117.5) .. (205.03,117.5) .. controls (202.68,117.5) and (200.72,115.6) .. (200.64,113.25) -- cycle ;
%Straight Lines [id:da5055160852316] 
\draw [color={rgb, 255:red, 208; green, 2; blue, 27 }  ,draw opacity=1 ]   (204.89,113.25) -- (204.77,122.99) -- (204.61,135.25) ;
%Shape: Circle [id:dp6917188966153562] 
\draw  [color={rgb, 255:red, 208; green, 2; blue, 27 }  ,draw opacity=1 ][fill={rgb, 255:red, 208; green, 2; blue, 27 }  ,fill opacity=1 ] (200.36,182.25) .. controls (200.28,179.9) and (202.12,178) .. (204.47,178) .. controls (206.82,178) and (208.78,179.9) .. (208.86,182.25) .. controls (208.94,184.6) and (207.1,186.5) .. (204.75,186.5) .. controls (202.4,186.5) and (200.44,184.6) .. (200.36,182.25) -- cycle ;
%Shape: Circle [id:dp9679422755589189] 
\draw  [color={rgb, 255:red, 208; green, 2; blue, 27 }  ,draw opacity=1 ][fill={rgb, 255:red, 208; green, 2; blue, 27 }  ,fill opacity=1 ] (200.64,160.25) .. controls (200.56,157.9) and (202.4,156) .. (204.75,156) .. controls (207.1,156) and (209.06,157.9) .. (209.14,160.25) .. controls (209.22,162.6) and (207.38,164.5) .. (205.03,164.5) .. controls (202.68,164.5) and (200.72,162.6) .. (200.64,160.25) -- cycle ;
%Straight Lines [id:da07818446018157521] 
\draw [color={rgb, 255:red, 208; green, 2; blue, 27 }  ,draw opacity=1 ]   (204.89,160.25) -- (204.77,169.99) -- (204.61,182.25) ;
%Straight Lines [id:da40694539361108417] 
\draw [color={rgb, 255:red, 208; green, 2; blue, 27 }  ,draw opacity=1 ]   (205.17,138.25) -- (205.04,147.99) -- (204.89,160.25) ;
%Straight Lines [id:da09904572478588214] 
\draw [color={rgb, 255:red, 208; green, 2; blue, 27 }  ,draw opacity=1 ]   (204.75,181.5) -- (204.63,191.24) -- (204.09,230.08) ;
%Straight Lines [id:da7727062363150107] 
\draw [color={rgb, 255:red, 208; green, 2; blue, 27 }  ,draw opacity=1 ]   (204.87,99.26) -- (204.75,109) -- (204.89,113.25) ;
%Straight Lines [id:da562872462790979] 
\draw    (145.5,232.75) -- (178.25,254.25) ;
%Straight Lines [id:da43502350514285915] 
\draw    (178.25,254.25) -- (203.25,234.25) ;
%Shape: Circle [id:dp2901951224573859] 
\draw  [color={rgb, 255:red, 208; green, 2; blue, 27 }  ,draw opacity=1 ][fill={rgb, 255:red, 208; green, 2; blue, 27 }  ,fill opacity=1 ] (174.36,126.25) .. controls (174.28,123.9) and (176.12,122) .. (178.47,122) .. controls (180.82,122) and (182.78,123.9) .. (182.86,126.25) .. controls (182.94,128.6) and (181.1,130.5) .. (178.75,130.5) .. controls (176.4,130.5) and (174.44,128.6) .. (174.36,126.25) -- cycle ;
%Shape: Circle [id:dp8903052393991012] 
\draw  [color={rgb, 255:red, 208; green, 2; blue, 27 }  ,draw opacity=1 ][fill={rgb, 255:red, 208; green, 2; blue, 27 }  ,fill opacity=1 ] (174.64,104.25) .. controls (174.56,101.9) and (176.4,100) .. (178.75,100) .. controls (181.1,100) and (183.06,101.9) .. (183.14,104.25) .. controls (183.22,106.6) and (181.38,108.5) .. (179.03,108.5) .. controls (176.68,108.5) and (174.72,106.6) .. (174.64,104.25) -- cycle ;
%Straight Lines [id:da7248211041547015] 
\draw [color={rgb, 255:red, 208; green, 2; blue, 27 }  ,draw opacity=1 ]   (178.89,104.25) -- (178.77,113.99) -- (178.61,126.25) ;
%Shape: Circle [id:dp7339821376085116] 
\draw  [color={rgb, 255:red, 208; green, 2; blue, 27 }  ,draw opacity=1 ][fill={rgb, 255:red, 208; green, 2; blue, 27 }  ,fill opacity=1 ] (174.36,173.25) .. controls (174.28,170.9) and (176.12,169) .. (178.47,169) .. controls (180.82,169) and (182.78,170.9) .. (182.86,173.25) .. controls (182.94,175.6) and (181.1,177.5) .. (178.75,177.5) .. controls (176.4,177.5) and (174.44,175.6) .. (174.36,173.25) -- cycle ;
%Shape: Circle [id:dp4577712953038968] 
\draw  [color={rgb, 255:red, 208; green, 2; blue, 27 }  ,draw opacity=1 ][fill={rgb, 255:red, 208; green, 2; blue, 27 }  ,fill opacity=1 ] (174.64,151.25) .. controls (174.56,148.9) and (176.4,147) .. (178.75,147) .. controls (181.1,147) and (183.06,148.9) .. (183.14,151.25) .. controls (183.22,153.6) and (181.38,155.5) .. (179.03,155.5) .. controls (176.68,155.5) and (174.72,153.6) .. (174.64,151.25) -- cycle ;
%Straight Lines [id:da17961226940840047] 
\draw [color={rgb, 255:red, 208; green, 2; blue, 27 }  ,draw opacity=1 ]   (178.89,151.25) -- (178.77,160.99) -- (178.61,173.25) ;
%Straight Lines [id:da5557286299555914] 
\draw [color={rgb, 255:red, 208; green, 2; blue, 27 }  ,draw opacity=1 ]   (179.17,129.25) -- (179.04,138.99) -- (178.89,151.25) ;
%Straight Lines [id:da2051782931502294] 
\draw [color={rgb, 255:red, 208; green, 2; blue, 27 }  ,draw opacity=1 ]   (178.75,172.5) -- (178.63,182.24) -- (178.25,250) ;
%Straight Lines [id:da09691156710929028] 
\draw [color={rgb, 255:red, 208; green, 2; blue, 27 }  ,draw opacity=1 ]   (178.87,90.26) -- (178.75,100) -- (178.89,104.25) ;
%Straight Lines [id:da38862451974733325] 
\draw    (182.5,254.25) -- (229.5,253) ;
%Shape: Circle [id:dp2698216631821587] 
\draw  [fill={rgb, 255:red, 0; green, 0; blue, 0 }  ,fill opacity=1 ] (225.25,253) .. controls (225.25,250.65) and (227.15,248.75) .. (229.5,248.75) .. controls (231.85,248.75) and (233.75,250.65) .. (233.75,253) .. controls (233.75,255.35) and (231.85,257.25) .. (229.5,257.25) .. controls (227.15,257.25) and (225.25,255.35) .. (225.25,253) -- cycle ;
%Shape: Circle [id:dp135885188231186] 
\draw  [color={rgb, 255:red, 208; green, 2; blue, 27 }  ,draw opacity=1 ][fill={rgb, 255:red, 208; green, 2; blue, 27 }  ,fill opacity=1 ] (226.36,149.25) .. controls (226.28,146.9) and (228.12,145) .. (230.47,145) .. controls (232.82,145) and (234.78,146.9) .. (234.86,149.25) .. controls (234.94,151.6) and (233.1,153.5) .. (230.75,153.5) .. controls (228.4,153.5) and (226.44,151.6) .. (226.36,149.25) -- cycle ;
%Shape: Circle [id:dp7485848083420028] 
\draw  [color={rgb, 255:red, 208; green, 2; blue, 27 }  ,draw opacity=1 ][fill={rgb, 255:red, 208; green, 2; blue, 27 }  ,fill opacity=1 ] (226.64,127.25) .. controls (226.56,124.9) and (228.4,123) .. (230.75,123) .. controls (233.1,123) and (235.06,124.9) .. (235.14,127.25) .. controls (235.22,129.6) and (233.38,131.5) .. (231.03,131.5) .. controls (228.68,131.5) and (226.72,129.6) .. (226.64,127.25) -- cycle ;
%Straight Lines [id:da03881610277000913] 
\draw [color={rgb, 255:red, 208; green, 2; blue, 27 }  ,draw opacity=1 ]   (230.89,127.25) -- (230.77,136.99) -- (230.61,149.25) ;
%Shape: Circle [id:dp706338506499921] 
\draw  [color={rgb, 255:red, 208; green, 2; blue, 27 }  ,draw opacity=1 ][fill={rgb, 255:red, 208; green, 2; blue, 27 }  ,fill opacity=1 ] (226.36,196.25) .. controls (226.28,193.9) and (228.12,192) .. (230.47,192) .. controls (232.82,192) and (234.78,193.9) .. (234.86,196.25) .. controls (234.94,198.6) and (233.1,200.5) .. (230.75,200.5) .. controls (228.4,200.5) and (226.44,198.6) .. (226.36,196.25) -- cycle ;
%Shape: Circle [id:dp2293242539680349] 
\draw  [color={rgb, 255:red, 208; green, 2; blue, 27 }  ,draw opacity=1 ][fill={rgb, 255:red, 208; green, 2; blue, 27 }  ,fill opacity=1 ] (226.64,174.25) .. controls (226.56,171.9) and (228.4,170) .. (230.75,170) .. controls (233.1,170) and (235.06,171.9) .. (235.14,174.25) .. controls (235.22,176.6) and (233.38,178.5) .. (231.03,178.5) .. controls (228.68,178.5) and (226.72,176.6) .. (226.64,174.25) -- cycle ;
%Straight Lines [id:da7481117528790571] 
\draw [color={rgb, 255:red, 208; green, 2; blue, 27 }  ,draw opacity=1 ]   (230.89,174.25) -- (230.77,183.99) -- (230.61,196.25) ;
%Straight Lines [id:da7273886219556006] 
\draw [color={rgb, 255:red, 208; green, 2; blue, 27 }  ,draw opacity=1 ]   (231.17,152.25) -- (231.04,161.99) -- (230.89,174.25) ;
%Straight Lines [id:da27641310988424983] 
\draw [color={rgb, 255:red, 208; green, 2; blue, 27 }  ,draw opacity=1 ]   (230.75,195.5) -- (230.63,205.24) -- (229.25,248.25) ;
%Straight Lines [id:da5403537501982357] 
\draw [color={rgb, 255:red, 208; green, 2; blue, 27 }  ,draw opacity=1 ]   (230.87,113.26) -- (230.75,123) -- (230.89,127.25) ;
%Straight Lines [id:da7695517408036888] 
\draw    (130.5,223) -- (143.25,232.75) ;
%Straight Lines [id:da4332866592471838] 
\draw    (145.5,234) -- (126.5,235) ;
%Straight Lines [id:da2554270240512111] 
\draw    (245.83,242) -- (231.03,251.75) ;
%Straight Lines [id:da19783945677170034] 
\draw    (234.5,253) -- (251.5,253) ;
%Straight Lines [id:da07563988688620216] 
\draw    (222.22,234.34) -- (207.42,235.09) ;
%Straight Lines [id:da9167560035621846] 
\draw    (218.9,224.33) -- (204.09,234.08) ;

\end{tikzpicture}
}  
 \resizebox{.3\textwidth}{!}{

\tikzset{every picture/.style={line width=0.75pt}} %set default line width to 0.75pt        

\begin{tikzpicture}[x=0.75pt,y=0.75pt,yscale=-1,xscale=1]

%Shape: Ellipse [id:dp563505254629851] 
\draw   (100,142.5) .. controls (100,121.24) and (151.71,104) .. (215.5,104) .. controls (279.29,104) and (331,121.24) .. (331,142.5) .. controls (331,163.76) and (279.29,181) .. (215.5,181) .. controls (151.71,181) and (100,163.76) .. (100,142.5) -- cycle ;
%Straight Lines [id:da8932565268917714] 
\draw    (141,141) -- (293,142) ;
%Straight Lines [id:da8868219193487027] 
\draw    (217,141.5) -- (256,120) ;
%Straight Lines [id:da6289160220197462] 
\draw    (215.5,141.5) -- (176,120) ;
%Straight Lines [id:da4800501356484894] 
\draw    (256,120) -- (265,127) ;
%Straight Lines [id:da10921562449563427] 
\draw    (141,141) -- (132,150) ;
%Straight Lines [id:da9487712067441404] 
\draw    (176,120) -- (168,128) ;
%Straight Lines [id:da3929343906669486] 
\draw    (176,120) -- (164,120) ;
%Straight Lines [id:da3767152580027807] 
\draw    (268,120) -- (256,120) ;
%Straight Lines [id:da552679114497161] 
\draw    (265,114) -- (256,120) ;
%Straight Lines [id:da6726066325622646] 
\draw    (293,142) -- (302,149) ;
%Straight Lines [id:da6769829973881446] 
\draw    (168,112) -- (176,120) ;
%Straight Lines [id:da6104497651784522] 
\draw    (141,141) -- (129,141) ;
%Straight Lines [id:da7880503217083507] 
\draw    (132,133) -- (141,141) ;
%Straight Lines [id:da9574927930870788] 
\draw    (305,142) -- (293,142) ;
%Straight Lines [id:da9220829124230346] 
\draw    (301,135) -- (293,142) ;
%Shape: Circle [id:dp2531828095976161] 
\draw  [fill={rgb, 255:red, 0; green, 0; blue, 0 }  ,fill opacity=1 ] (136,141) .. controls (136,139.62) and (137.12,138.5) .. (138.5,138.5) .. controls (139.88,138.5) and (141,139.62) .. (141,141) .. controls (141,142.38) and (139.88,143.5) .. (138.5,143.5) .. controls (137.12,143.5) and (136,142.38) .. (136,141) -- cycle ;
%Shape: Circle [id:dp5082144509238644] 
\draw  [fill={rgb, 255:red, 0; green, 0; blue, 0 }  ,fill opacity=1 ] (173.5,120.5) .. controls (173.5,119.12) and (174.62,118) .. (176,118) .. controls (177.38,118) and (178.5,119.12) .. (178.5,120.5) .. controls (178.5,121.88) and (177.38,123) .. (176,123) .. controls (174.62,123) and (173.5,121.88) .. (173.5,120.5) -- cycle ;
%Shape: Circle [id:dp6792072243878619] 
\draw  [fill={rgb, 255:red, 0; green, 0; blue, 0 }  ,fill opacity=1 ] (213,142) .. controls (213,140.62) and (214.12,139.5) .. (215.5,139.5) .. controls (216.88,139.5) and (218,140.62) .. (218,142) .. controls (218,143.38) and (216.88,144.5) .. (215.5,144.5) .. controls (214.12,144.5) and (213,143.38) .. (213,142) -- cycle ;
%Shape: Circle [id:dp9846260780640548] 
\draw  [fill={rgb, 255:red, 0; green, 0; blue, 0 }  ,fill opacity=1 ] (253.5,120.5) .. controls (253.5,119.12) and (254.62,118) .. (256,118) .. controls (257.38,118) and (258.5,119.12) .. (258.5,120.5) .. controls (258.5,121.88) and (257.38,123) .. (256,123) .. controls (254.62,123) and (253.5,121.88) .. (253.5,120.5) -- cycle ;
%Shape: Circle [id:dp7349169862121815] 
\draw  [fill={rgb, 255:red, 0; green, 0; blue, 0 }  ,fill opacity=1 ] (290.5,142.5) .. controls (290.5,141.12) and (291.62,140) .. (293,140) .. controls (294.38,140) and (295.5,141.12) .. (295.5,142.5) .. controls (295.5,143.88) and (294.38,145) .. (293,145) .. controls (291.62,145) and (290.5,143.88) .. (290.5,142.5) -- cycle ;
%Straight Lines [id:da13710192491263373] 
\draw [color={rgb, 255:red, 208; green, 2; blue, 27 }  ,draw opacity=1 ][fill={rgb, 255:red, 208; green, 2; blue, 27 }  ,fill opacity=1 ]   (139,85) -- (139,138) ;
%Shape: Circle [id:dp7977019659191733] 
\draw  [color={rgb, 255:red, 208; green, 2; blue, 27 }  ,draw opacity=1 ][fill={rgb, 255:red, 208; green, 2; blue, 27 }  ,fill opacity=1 ] (136.5,87.5) .. controls (136.5,86.12) and (137.62,85) .. (139,85) .. controls (140.38,85) and (141.5,86.12) .. (141.5,87.5) .. controls (141.5,88.88) and (140.38,90) .. (139,90) .. controls (137.62,90) and (136.5,88.88) .. (136.5,87.5) -- cycle ;
%Straight Lines [id:da43325447634944714] 
\draw [color={rgb, 255:red, 208; green, 2; blue, 27 }  ,draw opacity=1 ][fill={rgb, 255:red, 208; green, 2; blue, 27 }  ,fill opacity=1 ]   (152,71) -- (139,87.5) ;
%Shape: Circle [id:dp6009472961221994] 
\draw  [color={rgb, 255:red, 208; green, 2; blue, 27 }  ,draw opacity=1 ][fill={rgb, 255:red, 208; green, 2; blue, 27 }  ,fill opacity=1 ] (148.5,72.5) .. controls (148.5,71.12) and (149.62,70) .. (151,70) .. controls (152.38,70) and (153.5,71.12) .. (153.5,72.5) .. controls (153.5,73.88) and (152.38,75) .. (151,75) .. controls (149.62,75) and (148.5,73.88) .. (148.5,72.5) -- cycle ;
%Straight Lines [id:da2133299853053403] 
\draw [color={rgb, 255:red, 208; green, 2; blue, 27 }  ,draw opacity=1 ][fill={rgb, 255:red, 208; green, 2; blue, 27 }  ,fill opacity=1 ]   (127,72) -- (139,87.5) ;
%Straight Lines [id:da3791437349205592] 
\draw [color={rgb, 255:red, 208; green, 2; blue, 27 }  ,draw opacity=1 ][fill={rgb, 255:red, 208; green, 2; blue, 27 }  ,fill opacity=1 ]   (132,59) -- (127,72) ;
%Shape: Circle [id:dp030252049053400287] 
\draw  [color={rgb, 255:red, 208; green, 2; blue, 27 }  ,draw opacity=1 ][fill={rgb, 255:red, 208; green, 2; blue, 27 }  ,fill opacity=1 ] (125,72) .. controls (125,70.62) and (126.12,69.5) .. (127.5,69.5) .. controls (128.88,69.5) and (130,70.62) .. (130,72) .. controls (130,73.38) and (128.88,74.5) .. (127.5,74.5) .. controls (126.12,74.5) and (125,73.38) .. (125,72) -- cycle ;
%Straight Lines [id:da000239746461372925] 
\draw [color={rgb, 255:red, 208; green, 2; blue, 27 }  ,draw opacity=1 ][fill={rgb, 255:red, 208; green, 2; blue, 27 }  ,fill opacity=1 ]   (122,59) -- (127.5,74.5) ;
%Straight Lines [id:da9629443194663311] 
\draw [color={rgb, 255:red, 208; green, 2; blue, 27 }  ,draw opacity=1 ][fill={rgb, 255:red, 208; green, 2; blue, 27 }  ,fill opacity=1 ]   (157,58) -- (152,71) ;
%Straight Lines [id:da5346882121613288] 
\draw [color={rgb, 255:red, 208; green, 2; blue, 27 }  ,draw opacity=1 ][fill={rgb, 255:red, 208; green, 2; blue, 27 }  ,fill opacity=1 ]   (147,58) -- (152.5,73.5) ;
%Straight Lines [id:da7810697168461476] 
\draw [color={rgb, 255:red, 208; green, 2; blue, 27 }  ,draw opacity=1 ][fill={rgb, 255:red, 208; green, 2; blue, 27 }  ,fill opacity=1 ]   (176,65) -- (176,118) ;
%Shape: Circle [id:dp8469931490030765] 
\draw  [color={rgb, 255:red, 208; green, 2; blue, 27 }  ,draw opacity=1 ][fill={rgb, 255:red, 208; green, 2; blue, 27 }  ,fill opacity=1 ] (173.5,67.5) .. controls (173.5,66.12) and (174.62,65) .. (176,65) .. controls (177.38,65) and (178.5,66.12) .. (178.5,67.5) .. controls (178.5,68.88) and (177.38,70) .. (176,70) .. controls (174.62,70) and (173.5,68.88) .. (173.5,67.5) -- cycle ;
%Straight Lines [id:da3940146091371416] 
\draw [color={rgb, 255:red, 208; green, 2; blue, 27 }  ,draw opacity=1 ][fill={rgb, 255:red, 208; green, 2; blue, 27 }  ,fill opacity=1 ]   (189,51) -- (176,67.5) ;
%Shape: Circle [id:dp044938874997583644] 
\draw  [color={rgb, 255:red, 208; green, 2; blue, 27 }  ,draw opacity=1 ][fill={rgb, 255:red, 208; green, 2; blue, 27 }  ,fill opacity=1 ] (185.5,52.5) .. controls (185.5,51.12) and (186.62,50) .. (188,50) .. controls (189.38,50) and (190.5,51.12) .. (190.5,52.5) .. controls (190.5,53.88) and (189.38,55) .. (188,55) .. controls (186.62,55) and (185.5,53.88) .. (185.5,52.5) -- cycle ;
%Straight Lines [id:da5586736697544931] 
\draw [color={rgb, 255:red, 208; green, 2; blue, 27 }  ,draw opacity=1 ][fill={rgb, 255:red, 208; green, 2; blue, 27 }  ,fill opacity=1 ]   (164,52) -- (176,67.5) ;
%Straight Lines [id:da9082960107400596] 
\draw [color={rgb, 255:red, 208; green, 2; blue, 27 }  ,draw opacity=1 ][fill={rgb, 255:red, 208; green, 2; blue, 27 }  ,fill opacity=1 ]   (169,39) -- (164,52) ;
%Shape: Circle [id:dp08147998019307134] 
\draw  [color={rgb, 255:red, 208; green, 2; blue, 27 }  ,draw opacity=1 ][fill={rgb, 255:red, 208; green, 2; blue, 27 }  ,fill opacity=1 ] (162,52) .. controls (162,50.62) and (163.12,49.5) .. (164.5,49.5) .. controls (165.88,49.5) and (167,50.62) .. (167,52) .. controls (167,53.38) and (165.88,54.5) .. (164.5,54.5) .. controls (163.12,54.5) and (162,53.38) .. (162,52) -- cycle ;
%Straight Lines [id:da6732244741558446] 
\draw [color={rgb, 255:red, 208; green, 2; blue, 27 }  ,draw opacity=1 ][fill={rgb, 255:red, 208; green, 2; blue, 27 }  ,fill opacity=1 ]   (159,39) -- (164.5,54.5) ;
%Straight Lines [id:da4468328324205092] 
\draw [color={rgb, 255:red, 208; green, 2; blue, 27 }  ,draw opacity=1 ][fill={rgb, 255:red, 208; green, 2; blue, 27 }  ,fill opacity=1 ]   (194,38) -- (189,51) ;
%Straight Lines [id:da5700090502875389] 
\draw [color={rgb, 255:red, 208; green, 2; blue, 27 }  ,draw opacity=1 ][fill={rgb, 255:red, 208; green, 2; blue, 27 }  ,fill opacity=1 ]   (184,38) -- (189.5,53.5) ;
%Straight Lines [id:da712200826467607] 
\draw [color={rgb, 255:red, 208; green, 2; blue, 27 }  ,draw opacity=1 ][fill={rgb, 255:red, 208; green, 2; blue, 27 }  ,fill opacity=1 ]   (215.5,86) -- (215.5,139) ;
%Shape: Circle [id:dp5451281857171769] 
\draw  [color={rgb, 255:red, 208; green, 2; blue, 27 }  ,draw opacity=1 ][fill={rgb, 255:red, 208; green, 2; blue, 27 }  ,fill opacity=1 ] (213,88.5) .. controls (213,87.12) and (214.12,86) .. (215.5,86) .. controls (216.88,86) and (218,87.12) .. (218,88.5) .. controls (218,89.88) and (216.88,91) .. (215.5,91) .. controls (214.12,91) and (213,89.88) .. (213,88.5) -- cycle ;
%Straight Lines [id:da5009187840447745] 
\draw [color={rgb, 255:red, 208; green, 2; blue, 27 }  ,draw opacity=1 ][fill={rgb, 255:red, 208; green, 2; blue, 27 }  ,fill opacity=1 ]   (228.5,72) -- (215.5,88.5) ;
%Shape: Circle [id:dp6110389686039585] 
\draw  [color={rgb, 255:red, 208; green, 2; blue, 27 }  ,draw opacity=1 ][fill={rgb, 255:red, 208; green, 2; blue, 27 }  ,fill opacity=1 ] (225,73.5) .. controls (225,72.12) and (226.12,71) .. (227.5,71) .. controls (228.88,71) and (230,72.12) .. (230,73.5) .. controls (230,74.88) and (228.88,76) .. (227.5,76) .. controls (226.12,76) and (225,74.88) .. (225,73.5) -- cycle ;
%Straight Lines [id:da7711005296723301] 
\draw [color={rgb, 255:red, 208; green, 2; blue, 27 }  ,draw opacity=1 ][fill={rgb, 255:red, 208; green, 2; blue, 27 }  ,fill opacity=1 ]   (203.5,73) -- (215.5,88.5) ;
%Straight Lines [id:da557152147541144] 
\draw [color={rgb, 255:red, 208; green, 2; blue, 27 }  ,draw opacity=1 ][fill={rgb, 255:red, 208; green, 2; blue, 27 }  ,fill opacity=1 ]   (208.5,60) -- (203.5,73) ;
%Shape: Circle [id:dp9435606159727102] 
\draw  [color={rgb, 255:red, 208; green, 2; blue, 27 }  ,draw opacity=1 ][fill={rgb, 255:red, 208; green, 2; blue, 27 }  ,fill opacity=1 ] (201.5,73) .. controls (201.5,71.62) and (202.62,70.5) .. (204,70.5) .. controls (205.38,70.5) and (206.5,71.62) .. (206.5,73) .. controls (206.5,74.38) and (205.38,75.5) .. (204,75.5) .. controls (202.62,75.5) and (201.5,74.38) .. (201.5,73) -- cycle ;
%Straight Lines [id:da8028856293953804] 
\draw [color={rgb, 255:red, 208; green, 2; blue, 27 }  ,draw opacity=1 ][fill={rgb, 255:red, 208; green, 2; blue, 27 }  ,fill opacity=1 ]   (198.5,60) -- (204,75.5) ;
%Straight Lines [id:da7321521232514194] 
\draw [color={rgb, 255:red, 208; green, 2; blue, 27 }  ,draw opacity=1 ][fill={rgb, 255:red, 208; green, 2; blue, 27 }  ,fill opacity=1 ]   (233.5,59) -- (228.5,72) ;
%Straight Lines [id:da6201527688758841] 
\draw [color={rgb, 255:red, 208; green, 2; blue, 27 }  ,draw opacity=1 ][fill={rgb, 255:red, 208; green, 2; blue, 27 }  ,fill opacity=1 ]   (223.5,59) -- (229,74.5) ;
%Straight Lines [id:da980264381802629] 
\draw [color={rgb, 255:red, 208; green, 2; blue, 27 }  ,draw opacity=1 ][fill={rgb, 255:red, 208; green, 2; blue, 27 }  ,fill opacity=1 ]   (256,65) -- (256,118) ;
%Shape: Circle [id:dp9568687118552273] 
\draw  [color={rgb, 255:red, 208; green, 2; blue, 27 }  ,draw opacity=1 ][fill={rgb, 255:red, 208; green, 2; blue, 27 }  ,fill opacity=1 ] (253.5,67.5) .. controls (253.5,66.12) and (254.62,65) .. (256,65) .. controls (257.38,65) and (258.5,66.12) .. (258.5,67.5) .. controls (258.5,68.88) and (257.38,70) .. (256,70) .. controls (254.62,70) and (253.5,68.88) .. (253.5,67.5) -- cycle ;
%Straight Lines [id:da7663468797076153] 
\draw [color={rgb, 255:red, 208; green, 2; blue, 27 }  ,draw opacity=1 ][fill={rgb, 255:red, 208; green, 2; blue, 27 }  ,fill opacity=1 ]   (269,51) -- (256,67.5) ;
%Shape: Circle [id:dp3380210801120702] 
\draw  [color={rgb, 255:red, 208; green, 2; blue, 27 }  ,draw opacity=1 ][fill={rgb, 255:red, 208; green, 2; blue, 27 }  ,fill opacity=1 ] (265.5,52.5) .. controls (265.5,51.12) and (266.62,50) .. (268,50) .. controls (269.38,50) and (270.5,51.12) .. (270.5,52.5) .. controls (270.5,53.88) and (269.38,55) .. (268,55) .. controls (266.62,55) and (265.5,53.88) .. (265.5,52.5) -- cycle ;
%Straight Lines [id:da37039066798226794] 
\draw [color={rgb, 255:red, 208; green, 2; blue, 27 }  ,draw opacity=1 ][fill={rgb, 255:red, 208; green, 2; blue, 27 }  ,fill opacity=1 ]   (244,52) -- (256,67.5) ;
%Straight Lines [id:da4383678547154314] 
\draw [color={rgb, 255:red, 208; green, 2; blue, 27 }  ,draw opacity=1 ][fill={rgb, 255:red, 208; green, 2; blue, 27 }  ,fill opacity=1 ]   (249,39) -- (244,52) ;
%Shape: Circle [id:dp21503388654232514] 
\draw  [color={rgb, 255:red, 208; green, 2; blue, 27 }  ,draw opacity=1 ][fill={rgb, 255:red, 208; green, 2; blue, 27 }  ,fill opacity=1 ] (242,52) .. controls (242,50.62) and (243.12,49.5) .. (244.5,49.5) .. controls (245.88,49.5) and (247,50.62) .. (247,52) .. controls (247,53.38) and (245.88,54.5) .. (244.5,54.5) .. controls (243.12,54.5) and (242,53.38) .. (242,52) -- cycle ;
%Straight Lines [id:da9353997830062813] 
\draw [color={rgb, 255:red, 208; green, 2; blue, 27 }  ,draw opacity=1 ][fill={rgb, 255:red, 208; green, 2; blue, 27 }  ,fill opacity=1 ]   (239,39) -- (244.5,54.5) ;
%Straight Lines [id:da6772534020058983] 
\draw [color={rgb, 255:red, 208; green, 2; blue, 27 }  ,draw opacity=1 ][fill={rgb, 255:red, 208; green, 2; blue, 27 }  ,fill opacity=1 ]   (274,38) -- (269,51) ;
%Straight Lines [id:da18834368375544863] 
\draw [color={rgb, 255:red, 208; green, 2; blue, 27 }  ,draw opacity=1 ][fill={rgb, 255:red, 208; green, 2; blue, 27 }  ,fill opacity=1 ]   (264,38) -- (269.5,53.5) ;
%Straight Lines [id:da7852184752857533] 
\draw [color={rgb, 255:red, 208; green, 2; blue, 27 }  ,draw opacity=1 ][fill={rgb, 255:red, 208; green, 2; blue, 27 }  ,fill opacity=1 ]   (293,87) -- (293,140) ;
%Shape: Circle [id:dp8537351559311236] 
\draw  [color={rgb, 255:red, 208; green, 2; blue, 27 }  ,draw opacity=1 ][fill={rgb, 255:red, 208; green, 2; blue, 27 }  ,fill opacity=1 ] (290.5,89.5) .. controls (290.5,88.12) and (291.62,87) .. (293,87) .. controls (294.38,87) and (295.5,88.12) .. (295.5,89.5) .. controls (295.5,90.88) and (294.38,92) .. (293,92) .. controls (291.62,92) and (290.5,90.88) .. (290.5,89.5) -- cycle ;
%Straight Lines [id:da9172443201915716] 
\draw [color={rgb, 255:red, 208; green, 2; blue, 27 }  ,draw opacity=1 ][fill={rgb, 255:red, 208; green, 2; blue, 27 }  ,fill opacity=1 ]   (306,73) -- (293,89.5) ;
%Shape: Circle [id:dp20691010866015624] 
\draw  [color={rgb, 255:red, 208; green, 2; blue, 27 }  ,draw opacity=1 ][fill={rgb, 255:red, 208; green, 2; blue, 27 }  ,fill opacity=1 ] (302.5,74.5) .. controls (302.5,73.12) and (303.62,72) .. (305,72) .. controls (306.38,72) and (307.5,73.12) .. (307.5,74.5) .. controls (307.5,75.88) and (306.38,77) .. (305,77) .. controls (303.62,77) and (302.5,75.88) .. (302.5,74.5) -- cycle ;
%Straight Lines [id:da23772078161075705] 
\draw [color={rgb, 255:red, 208; green, 2; blue, 27 }  ,draw opacity=1 ][fill={rgb, 255:red, 208; green, 2; blue, 27 }  ,fill opacity=1 ]   (281,74) -- (293,89.5) ;
%Straight Lines [id:da29252271701214483] 
\draw [color={rgb, 255:red, 208; green, 2; blue, 27 }  ,draw opacity=1 ][fill={rgb, 255:red, 208; green, 2; blue, 27 }  ,fill opacity=1 ]   (286,61) -- (281,74) ;
%Shape: Circle [id:dp3496688918003099] 
\draw  [color={rgb, 255:red, 208; green, 2; blue, 27 }  ,draw opacity=1 ][fill={rgb, 255:red, 208; green, 2; blue, 27 }  ,fill opacity=1 ] (279,74) .. controls (279,72.62) and (280.12,71.5) .. (281.5,71.5) .. controls (282.88,71.5) and (284,72.62) .. (284,74) .. controls (284,75.38) and (282.88,76.5) .. (281.5,76.5) .. controls (280.12,76.5) and (279,75.38) .. (279,74) -- cycle ;
%Straight Lines [id:da292976133917475] 
\draw [color={rgb, 255:red, 208; green, 2; blue, 27 }  ,draw opacity=1 ][fill={rgb, 255:red, 208; green, 2; blue, 27 }  ,fill opacity=1 ]   (276,61) -- (281.5,76.5) ;
%Straight Lines [id:da8095300484915079] 
\draw [color={rgb, 255:red, 208; green, 2; blue, 27 }  ,draw opacity=1 ][fill={rgb, 255:red, 208; green, 2; blue, 27 }  ,fill opacity=1 ]   (311,60) -- (306,73) ;
%Straight Lines [id:da5783875586285452] 
\draw [color={rgb, 255:red, 208; green, 2; blue, 27 }  ,draw opacity=1 ][fill={rgb, 255:red, 208; green, 2; blue, 27 }  ,fill opacity=1 ]   (301,60) -- (306.5,75.5) ;
\end{tikzpicture}
}

\caption{The trees $X(4,2)$  and $X(5,3)$.}\label{fig:examples}
\end{figure}
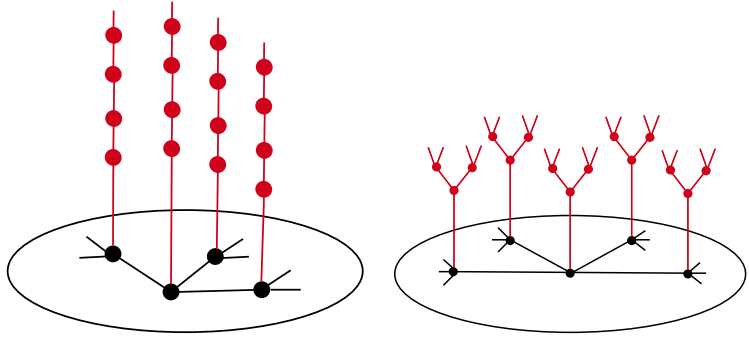

The cop strategy $S_c^\ast$ of reducing the distance between the cop and the robber in every sober cop move is also a sensible strategy on the trees $X(\Delta,\delta)$ (and in fact on any tree).
On the other hand, there are at least two intuitively sensible robber strategies on $X(\Delta,\delta)$. We denote these strategies by \RSA \ and \RSB. In both \RSA \ and \RSB, if the robber is in the base tree $B$, they increase the distance between the two players while staying in $B$.
In a small tree $T$, a robber playing strategy \RSA \ will keep increasing the distance from the cop by picking a neighbor of $\R_{\tim}$ which does not lie on the unique path connecting $\R_{\tim}$ to $\C_{\tim}$; among the possible choices they pick one uniformly at random. A robber playing strategy \RSB \ will opt to backtrack toward the base tree $B$, that is,  pick the unique  neighbor of $\R_{\tim}$ that lies closer to the base tree $B$ than $\R_{\tim}$.   
For different values of $\Delta$ and $\delta,$ we will see that there are certain scenarios in $X(\Delta,\delta)$ where \RSA \ is preferable to \RSB \ and vice versa. 
To simplify the analysis of the robber strategies we also introduce a cop strategy \CSB\ which is similar to \RSB: On the base tree $B$, a cop playing strategy \CSB\ moves towards the copy of $T$ in which the robber is located (and picks a uniform random neighbor in the base tree, if the cop and the robber are located in the same copy). If the cop is not located in the base tree $B$, then similarly to a robber playing \RSB\, they move to the unique  neighbor of $\C_{\tim}$ that lies closer to the base tree $B$ than $\C_{\tim}$. We point out that the only situation in which \CSB\ might differ from \CS\ is when the two players are on the same small tree $T$, and the  only advantage of \CSB\ over \CS\ is that it is easier to analyse against \RSB.

The main result of this section shows that it depends on the values of $\Delta$, $\delta$, $p_c^t$, and $p_r^t$, which of the two strategies \RSA\ and \RSB\ is better for the robber. We will give precise bounds on the tipsiness parameters for which the robber is guaranteed to win playing either of the two strategies, see Theorems \ref{thm:robbersober} and \ref{thm:RSB}. These bounds imply the following result.

\begin{theorem} \label{theorem:strategy_summary}
Consider the game on $X(\Delta,\delta)$ with $\Delta \geq 4$ and assume that $p_c^s+p_c^t = p_r^s + p_r^t= 1/2$. There is some $p_0 \in [0,\frac{\delta}{4(\delta-1)}]$ such that
\begin{enumerate}
    \item if $p_r^t \leq p_0$ and \RSA\ is winning against every cop strategy for some fixed $p_c^t$, then so is \RSB, and
    \item if $p_r^t > p_0$ and \RSB\ is winning against every cop strategy for some fixed $p_c^t$, then so is \RSA.
\end{enumerate}
If $\delta=2$, then $p_0 = \frac 12$. If $2 \delta \geq \Delta+1$, then $p_0 = 0$. Otherwise the value of $p_0$ lies strictly between $0$ and $\frac 12$ and can be determined by solving a quadratic equation whose coefficients depend on $\delta$ and $\Delta$.
\end{theorem}

We note that the restriction $\Delta \geq 4$ only rules out the infinite comb $X(3,2)$. %For $X(3,3)$, Lemma \ref{lemma:delta} tells us when each player has a winning strategy. 
For $X(3,2)$, the robber strategy \RSA\ is winning if $p_c^t > p_r^t$ by the same argument as in the proof of Lemma \ref{lem:robber_wins}. If $p_c^t < p_r^t$, then \CS\ is easily seen to be winning by a slight modification of the argument in the proof of the second part of Theorem~\ref{thm:robbersober}: the cop will almost surely get to the same small tree as the robber, and once this happens there is a positive probability that the cop will not leave this tree before catching the robber. In particular the only case for which Theorem~\ref{theorem:strategy_summary} does not provide an answer is not especially difficult.

\subsection{Analysis of \RSA\ and \RSB}

The goal of this section is to provide conditions for when \RSA\ and \RSB\ are winning on $X(\Delta,\delta)$. The main results of this section are Theorem \ref{thm:robbersober} which tells us exactly when \RSA\ is winning and Theorem \ref{thm:RSB} which tells us in almost all cases whether \RSB\ is winning.

We start by setting up the notation we use to analyse the game on $X(\Delta,\delta)$. Throughout our analysis, we keep track of the projections of the cop's and robber's position to the base tree, in other words we consider the auxiliary random process $(\C_{\tim}',\R_{\tim}')$ where $\C_{\tim}'$ is the vertex of $B$ which is closest to $\C_{\tim}$, and $\R_{\tim}'$ is the vertex of $B$ which is closest to $\R_{\tim}$. 
%Moreover, define $H_{c,\tim}$ as the distance between $\C_{\tim}$ and $\C_{\tim}'$, and $H_{r,\tim}$ as the distance between $\R_{\tim}$ and $\R_{\tim}'$; these are the height at which the cop and robber are in the $\delta$-regular tree. 
Further denote by $T_c(n)$ the (random) time step at which the cop makes their $n$-th move on the base tree $B$, that is, $T_c(0) = \inf\{\tim \geq 0 \mid \C_{\tim-1}' \neq \C_{\tim}'\}$, and $T_c(n+1) = \inf\{\tim > T_c(n) \mid \C_{\tim-1}' \neq \C_{\tim}'\}$. 
Moreover, whenever $T_c(n) < \infty$,  we introduce an additional random variable $F_c(n)$ which take the value~$1$ if the move at time $T_c(n)$ increases the distance between the cop and the robber, and $-1$ if it decreases the distance. Analogously, define $T_r(n)$ and $F_r(n)$ for the robber.

Our analysis of the game is based on the total change in distance between $\C_{\tim}'$ and $\R_{\tim}'$ up to time step $\tim$, which is equal to
\begin{align}
Y_{\tim}:=\sum_{n: T_c(n) < \tim} F_c(n) + \sum_{n: T_r(n) < \tim} F_r(n).\label{eq1}
\end{align} 
The basic idea behind our analysis is that if $Y_{\tim}$ is positive for every $\tim$, then the distance between $\C_{\tim}'$ and $\R_{\tim}'$ and thus also the distance between $\C_{\tim}$ and $\R_{\tim}$ never reaches zero. On the other hand, if there are infinitely many time steps $\tim$ such that $Y_{\tim}=0$, then there are infinitely many time steps $\tim$ such that $\C_{\tim}'=\R_{\tim}'$. If for infinitely many of these steps the distance of both of the two players to the base tree $B$ is bounded, then the cop almost surely wins the game because every time this happens there is a positive probability that the cop wins the game within some bounded number of moves.

In order to bound $Y_{\tim}$ it will be useful to bound the individual summands in Equation~\eqref{eq1} by two auxiliary sequences of i.i.d.\ Bernoulli random variables. 

To define these random variables, we first restrict the allowed enumerations of neighbors in the construction of $\Q_{\tim}$. Let $(u,v)$ be a pair of vertices, where $u$ lies in the base tree. We assume that the enumeration of the neighbors of $u$ with respect to a pair $(u,v)$ is always chosen such that $u_1$ is a neighbor of $u$ in the base tree which minimizes the distance to $v$. Note that if $v$ lies in the small tree attached to $u$, then $u_1$ can be any neighbor of $u$ in the base tree, otherwise this rule uniquely determines $u_1$. Moreover, we let the last element $u_\Delta$ in the enumeration be the unique neighbor of $u$ in the small tree attached to $u$.

We now define
\[
    \tilde F_{c}(n) =
    \begin{cases}
        -1&\text{if } M_{T_{c}(n)} = \mathrm{cs} , \text{ or } M_{T_{c}(n)} = \mathrm{ct}  \text{ and } X_{T_{c}(n)} < \frac{1}{\Delta},\\
        1&\text{otherwise.}
    \end{cases}
\]
In particular, $\tilde F_{c}(n) = 1$ if $T_c(n) = \infty$. 

Letting $A_{\tim}$ denote the event that the cop's position at time step $\tim-1$ is in the base tree,  we compute
\begin{align*}
   \Prob[\tilde F_c(n)=-1 \mid (T_c(n) = \tim)]
   &= \Prob[\tilde F_c(n)=-1 \mid  T_c(n) = \tim \land A_\tim] \\
   &= \frac{\Prob[\tilde F_c(n)=-1 \land T_c(n) = \tim \mid A_\tim]}{\Prob[ T_c(n) = \tim \mid A_\tim]} \\
   %&= \frac{\Prob[M_{\tim} = \mathrm{cs}  \lor (M_{\tim} = \mathrm{ct}  \land X_{\tim} \geq \frac 1\Delta)]}{\Prob[M_{\tim} = (c,r) \lor (M_{\tim} = \mathrm{ct}  \land X_{\tim} < \frac{\Delta-1}{\Delta}]}\\
   &=\frac{p_c^s + \frac{1}{\Delta}p_c^t}{p_c^s-\frac{(\Delta-1)}{\Delta}p_c^t} \\
   &= \frac{1-\frac{2\Delta-2}{\Delta}p_c^t}{1 - \frac 2\Delta p_c^t}.
\end{align*}
Note that this probability does not depend on $\tim$, hence 
\[
\Prob[\tilde F_c(n)=-1 \mid T_c(n) < \infty] = \frac{1-\frac{2\Delta-2}{\Delta}p_c^t}{1 - \frac 2\Delta p_c^t}
\] 
and consequently, 
\[
     \E[\tilde F_c(n) \mid T_c(n) < \infty] = - \frac{1 - \frac{4\Delta-6}{\Delta}p_c^t}{1 - \frac{2}{\Delta}p_c^t}.
\]
We also note that conditional on the event that $T_c(n_0) < \infty$, the random variables $\tilde F_c(n)$ for $n \leq n_0$ are independent.

Similarly, we define independent random variables $\tilde F_r(n)$ for the robber. Again, we first restrict the allowed enumerations of neighbors in the construction of $\Q_{\tim}$. Let $(u,v)$ be a pair of vertices, where $v$ lies in the base tree. We assume that the enumeration of the neighbors of $v$ with respect to a pair $(u,v)$ is always chosen such that $v_{\Delta-1}$ is a neighbor of $v$ in the base tree which minimizes the distance to $u$. Moreover, we let $v_\Delta$ be the unique neighbor of $v$ in the small tree attached to $v$.

We now define
\[
    \tilde F_{r}(n) =
    \begin{cases}
        1&\text{if } M_{T_{c}(n)} = \mathrm{rs} , \text{ or } M_{T_{c}(n)} = \mathrm{rt}  \text{ and } X_{T_{c}(n)} < \frac{\Delta-2}{\Delta},\\
        -1&\text{otherwise.}
    \end{cases}
\]

Analogous computations as above yield
\[
\Prob[\tilde F_r(n)=1 \mid T_r(n) < \infty]  
= \frac{1 - \frac{4}{\Delta} p_r^t}{1 - \frac{2}{\Delta} p_r^t},
\]
and consequently
\[
    \E[\tilde F_r(n)\mid T_r(n) < \infty] =\frac{1 - \frac{6}{\Delta} p_r^t}{1 - \frac{2}{\Delta} p_r^t}.
\]
%Clearly, the variables $\tilde F_r(n)$ are i.i.d., they only depend on the outcome of the spinner.
We point out that conditional on the event $T_c(n_0) < \infty$ and $T_r(n_0) < \infty$, the random variables $\tilde F_c(n)$ and $\tilde F_r(n)$  for $n \leq n_0$ are mutually independent.

If $T_c(n)$ is finite then $F_c(n) \geq \tilde F_c(n)$ with equality in case the cop plays \CS\ or \CSB\ and $\C_{\tim}' \neq \R_{\tim}'$ at time step $\tim = T_c(n)$.
If the robber plays \RSA\ or  \RSB, and $T_c(n)$ is finite, then $F_r(n) \geq \tilde F_r(n)$ and equality holds if and only if $\C_{\tim}' \neq \R_{\tim}'$ at time step $\tim = T_r(n)$. Moreover, note that $\E[\tilde F_r(n)\mid T_r(n) < \infty]$ is larger than $0$ unless $\Delta = 3$ and $p_r^t = \frac 12$.

% Analogously, we can define a sequence of  i.i.d.\ random variables $\tilde F_c(n)$ taking values in $\{1,-1\}$ with $F_c(n) \geq \tilde F_c(n)$ for any cop strategy: $\tilde F_c(n)$ takes the value $-1$ if a sober cop move occurs at time step $\tim = T_c(n)$, and is otherwise defined analogously to $\tilde F_r(n)$. We note that
% \[
% \Prob[\tilde F_c(n)=1] = \frac{\frac{p_c^t(\Delta-2)}{\Delta}}{\frac{1}{2}-\frac{p_c^t}{\Delta}} 
% \]
% and thus
% \[
%      \E[\tilde F_c(n)] = - \frac{1 - \frac{4\Delta-6}{\Delta}p_c^t}{1 - \frac{2}{\Delta}p_c^t}.
% \]

It follows from the above discussion that if the robber plays \RSA\ or \RSB, then
\begin{align}
Y_{\tim} \geq \sum_{n: T_c(n) < \tim} \tilde F_c(n) + \sum_{n: T_r(n) < \tim} \tilde F_r(n).\label{eq2}
\end{align} 
This bound has the advantage that the summands are i.i.d.; if the expected values of the summands are positive, then it immediately follows from Theorem \ref{thm:gamblersruin} that the probability that neither of the two sums will ever reach zero is positive, which in turn implies that the robber has a positive probability of winning the game. 

We start by showing that this is in particular the case if the cop's tipsiness parameter $p_c^t$ is large enough.

\begin{proposition}
\label{prp:coptootipsy}
Consider the game on $X(\Delta,\delta)$ with $\Delta \geq 4$. If $p_c^t > \frac{\delta}{4\delta-4}$ or $p_c^t > \frac{\Delta}{4\Delta-6}$ then both \RSA\ and \RSB\ are winning against any cop strategy.
\end{proposition}

\begin{proof}
Note that $\Prob[\tilde F_r(n)=1|T_r(n)<\infty]>\frac 12$, hence $\Prob[\forall \tim \colon \sum_{n:T_r(n) \leq \tim} \tilde F_r(n) \geq 0] > 0$ by Theorem \ref{thm:gamblersruin}. Since the random variables $\tilde F_r(n)$ and $\tilde F_c(n)$ for $n \leq n_0$ are independent conditional on $T_r(n_0)<\infty$ and $T_c(n_0)<\infty$, it suffices to show that $\Prob[\forall \tim \colon \sum_{n:T_c(n) \leq \tim} \tilde F_c(n) \geq 0] > 0$.

Let $Z_{\tim} = d(\C_{\tim},\C_{\tim}')$. If $p_c^t > \frac{\delta}{4\delta-4}$, then
\[
\Prob[Z_{\tim+1}>Y_{\tim} \mid Z_{\tim}\neq 0] \geq p_c^t\frac{\delta-1}{\delta} > \frac12-p_c^t) + p_c^t\frac{1}{\delta} \geq \Prob[Z_{\tim-1}>Y_{\tim} \mid Z_{\tim}\neq 0],
\]
and thus by Theorem \ref{thm:gamblersruin} the probability that the cop never moves along an edge of the base tree $B$ is positive. It follows that $\Prob[T_c(0) = \infty] > 0$ and thus $\Prob[\forall \tim \colon \sum_{n: T_c(n) < \tim} \tilde F_c(n)=0] >0$.

If $p_c^t > \frac{\Delta}{4\Delta-6}$, then a straightforward calculation shows that $\Prob[\tilde F_c(n) = 1 \mid T_c(n) < \infty]>\frac 12$,  hence $\Prob[\forall \tim \colon \sum_{n:T_c(n) \leq \tim} \tilde F_c(n) \geq 0] > 0$ by Theorem \ref{thm:gamblersruin}.
\end{proof}

We note that neither of the two bounds in Proposition~\ref{prp:coptootipsy} generally implies the other; which one is stronger depends on the values of  $\Delta$ and $\delta$.

Our next goal is to describe precisely when the robber can win by playing \RSA. It is not surprising that in this case the game essentially boils down to a game on the $\delta$-regular tree (assuming that the cop is able to get to the same copy of the $\delta$-regular tree as the robber).

\begin{theorem}
\label{thm:robbersober}
If $p_c^t > \min\{\frac{p_r^t}{\delta - 1}, \frac{\Delta}{4\Delta-6}\}$, then robber strategy \RSA\ is winning against every cop strategy.
If $p_c^t \leq \min\{\frac{p_r^t}{\delta - 1}, \frac{\Delta}{4\Delta-6}\}$, then \CS\ is winning against \RSA,
provided that $p_r^t < \frac 12$, and against \RSB\ provided that $p_r^t > \frac{\delta}{4\delta-4}$.
\end{theorem}

\begin{proof}
If $p_c^t > \frac{\Delta}{4\Delta-6}$, then the probability that the robber wins the game is positive by Proposition \ref{prp:coptootipsy}. If $p_c^t > \frac{p_r^t}{\delta - 1}$ and the robber plays strategy \RSA, then the probability of increasing the distance between the cop and the robber in any move is at least
\[
    p_r^s+\frac{p_r^t(\delta-1)}{\delta} + \frac{p_c^t(\delta-1)}{\delta} \geq  \frac 12 - p_r^t + \frac{p_r^t(\delta-1)}{\delta} + \frac{p_r^t}{\delta}+ \epsilon = \frac 12 + \epsilon
\]
for some $\epsilon > 0$.
%Using appropriate vertex enumerations in the construction of $\Q_{\tim}$, we see that the distance can be bounded from below by a transient random walk on $\mathbb N$ depending only on $M_\tim$ and $X_\tim$. 
Hence by Proposition \ref{prp:inhomogeneous_recurrent} the probability that the distance never reaches $0$ is positive, and thus \RSA\ is winning.

Now assume that $p_c^t \leq \min\{\frac{p_r^t}{\delta - 1}, \frac{\Delta}{4\Delta-6}\}$ and either $p_r^t < \frac 12$ and the robber plays \RSA, or $p_r^t > \frac{\delta}{4\delta-4}$ and the robber plays \RSB. Note that if $d(R_{\tim},B) > 0$, then $R_{\tim}$ has $\delta - 1$ neighbours which lie further away from $B$ than $R_{\tim}$, and one neighbour which lies closer to $B$. An easy computation shows that if $\delta > 2$, then there are constants $p$ and $q$ such that
\[
    \Prob[d(R_{\tim},B)=d(R_{\tim-1},B)+1\mid d(R_{\tim},B) > 0] \geq p > q \geq \Prob[d(R_{\tim},B)=d(R_{\tim-1},B)-1\mid d(R_{\tim},B)> 0],
\]
in other words, when the robber is inside a small tree, then their probability of moving further away from the base tree is bigger than their probability of moving closer to the base tree. 

If $\delta = 2$, this is also true unless the robber plays \RSA and $R_{\tim}$ lies on the shortest path between $C_{\tim}$ and $B$. In this final case, it is not hard to see that the game will almost surely reach a state where $d(R_{\tim},B) > 0$ and $R_{\tim}$ does not lie on the shortest path between $C_{\tim}$ and $B$ (unless the robber gets caught before). Moreover, the expected number of steps before reaching such a state is finite.

In all cases, Lemma \ref{lem:transientonN} implies that there is some $k>0$ such that
\[\Prob[\exists\tim_0\forall \tim > \tim_0\colon d(\R_{\tim},B) > k \cdot \tim ] =1,
\]
and in particular the robber eventually stays in one small tree which we denote by $T_0$. If the cop almost surely eventually stays in $T_0$ as well, then the game reduces to a game on a $\delta$-regular tree which the cop almost surely wins by Lemma \ref{lemma:delta}. 

Now let $A$ be the event that there are infinitely many $\tim$ such that $\C_{\tim} \notin T_0$, and assume that $\Prob[A \mid \forall \tim > 0\colon d(\R_{\tim},B) > k \cdot \tim] > 0$. Whenever $\C_{\tim} \notin T_0$, the probability of reducing the distance $d(\C_{\tim},T_0)$ is greater or equal to the probability of increasing this distance because $p_c^t \leq \min\{\frac{p_r^t}{\delta - 1}, \frac{\Delta}{4\Delta-6}\} \leq \min\{\frac{\delta}{4\delta - 4}, \frac{\Delta}{4\Delta-6}\}$, where the second inequality follows from $p_r^t \leq \frac{1}{2}$ and $\delta \geq 2$. Hence (conditional on $A$) there are almost surely infinitely many $\tim$ for which $\C_{\tim}$ is equal to the unique vertex $b$ of $B$ which is adjacent to the root of $T_0$; let $T_n(b)$ be the $n$-th time the cop is at $b$, that is, $T_0(b) = \inf\{\tim\geq 0 \colon \C_{\tim} = b\}$ and $T_n(b) = \inf\{\tim >T_{n-1}(b) \colon \C_{\tim} = b\}$ for $n >0$. Moreover denote by $T_\dagger = \inf\{\tim \colon \C_{\tim} = \R_{\tim}\}$. Let $r$ be a vertex of $T_0$ whose distance to $B$ is at least $n$, and denote by $c$ the $n$-th vertex on the unique path from $b$ to $r$.

We claim that there is some $p > 0$ which does not depend on the precise choice of $r$ such that $\Prob_{(b,r)}[T_\dagger < T_1(b)] > p$. This clearly implies that $\Prob_{(b,r)}[T_\dagger > T_n(b)] <  (1-p)^n$, and thus  $\Prob[T_\dagger  < \infty \mid A] =1$. It remains to prove the claim.

We start by noting that there is a non-zero chance that the first $n$ moves are sober cop moves. Hence
\begin{align*}
    \Prob_{(b,r)}[T_\dagger < T_1(b)]
    &\geq \Prob_{(b,r)}[(\C_{n},\R_{n}) = (c,r)] \cdot \Prob_{(c,r)}[T_\dagger < T_1(b)]\\
    &\geq (p_c^s)^{n} \cdot \Prob_{(c,r)}[\forall \tim \geq 0\colon d(\R_{\tim},b) > d(r,b)-n \land T_\dagger < \inf\{\tim \colon d(\C_{\tim},\R_{\tim}) > d(c,r)\}].
\end{align*}

Conditional on both players' positions being inside $T_0$, it is more likely to decrease the distance between them than to increase that distance.
More precisely, there are constants $p$ and $q$ such that 
\[
    \Prob[d(\C_{\tim},\R_{\tim}) \leq d(\C_{\tim-1},\R_{\tim-1})+1 \mid \C_{\tim}, \R_{\tim} \in T_0] = p < q= \Prob[d(\C_{\tim},\R_{\tim}) \leq d(\C_{\tim-1},\R_{\tim-1})-1 \mid \C_{\tim}, \R_{\tim} \in T_0] .
\]
Therefore, by Proposition \ref{prp:zerobeforeincrease} there is $\epsilon > 0$ such that
\[
\Prob_{(c,r)}[T_\dagger < \inf\{\tim \colon d(\C_{\tim},\R_{\tim}) > d(c,r)\} \mid \forall \tim \geq 0\colon d(\R_{\tim},b) > d(r,b)-n] > \epsilon.
\]
Finally recall that the probability that a robber move in a small tree increases the distance between the robber's position and the base tree is larger than the probability that it decreases this distance. Thus, if $n$ is large enough, then $\Prob_{(c,r)}[\forall \tim \geq 0\colon d(\R_{\tim},b) > d(r,b)-n] > 0$ by Lemma \ref{lem:transientonN}. 
%
% In order to show that the cop almost surely eventually remains in $T_0$, it suffices to show that every time the cop moves from $B$ to the root of $T_0$, the probability of them never visiting $B$ again is at least some $\epsilon > 0$.
%
% As noted above, we may condition on the event that the distance between $\R_{\tim}$ and $B$ is at least $k \cdot \tim$ for every $\tim$.
%
% Denote by $P_{\tim}$ the unique path connecting the root of $T_0$ to $\R_{\tim}$, and denote by $\C_{\tim}''$ the vertex closest to $\C_{\tim}$ on this path. If $\C_{\tim} = \C_{\tim}''$, then the probability of increasing the distance between $\C_{\tim}''$ and $B$ is $c+\frac{p_c^t}{\delta}$, the probability of decreasing this distance is $\frac{p_c^t}{\delta}$, and the probability of the distance remaining unchanged is $\frac{p_c^t(\delta-2)}{\delta}$.
% If $\C_{\tim} \neq \C_{\tim}''$ then the distance can only change if $\R_{\tim} = \C_{\tim}''$, in which case the distance is at least $k \cdot \tim$.
% Hence the distance between $\C'_{\tim}$ and $B$ is a random walk on $\mathbb N$ whose probability of moving away from $0$ is strictly larger than its probability of moving towards $0$, and thus it has a positive probability of never returning to $0$.
\end{proof}

In the remainder of this subsection we analyse \RSB. The case $p_r^t > \frac{\delta}{4\delta-4}$ where the robber is too tipsy to get back to the base tree an infinite number of times is already covered in the previous theorem; we may thus focus on $p_r^t \leq \frac{\delta}{4\delta-4}$. We start by analyzing the number of steps in between consecutive cop and robber moves on the base tree.

\begin{lemma}\label{lemma:stationary_time}
Consider the game on $X(\Delta,\delta)$. For $p \in \{p_c^t$, $p_r^t\}$ define
\[
\mu(p) = 
    \frac 12 \cdot \frac{
    \left(1-\frac{2}{\Delta} p\right)\left(1- \frac{4\delta-4}{\delta} p\right)
    }
    {
    1 - \left(\frac{4 \delta - 4}{\delta}- \frac{2}{\Delta} \right)p
    }
\]

If $p_r^t \leq \frac{\delta}{4\delta-4}$, then for every $\epsilon >0$ there are positive constants $a$ and $b$ such that (regardless of the robber's strategy)
\[
    \Prob[\max\{n\colon T_r(n) \leq \tim\} \leq  (\mu(p_r^t) + \epsilon) \tim] > 1 - a e^{-b\tim}.
\]
If $p_r^t < \frac{\delta}{4\delta-4}$ and the robber plays \RSB, then we also get 
\[
\Prob[\max\{n\colon T_r(n) \leq \tim\} \geq (\mu(p_r^t)-\epsilon) \tim] > 1 - a e^{-b\tim}.
\]

Similarly, if $p_c^t \leq \frac{\delta}{4\delta-4}$, then for every $\epsilon >0$ there are positive constants $a$ and $b$ such that (regardless of the cop's strategy)
\[
    \Prob[\max\{n\colon T_c(n) \leq \tim\} \leq (\mu(p_c^t)+\epsilon) \tim] > 1 - a e^{-b\tim}.
\]
If $p_r^t < \frac{\delta}{4\delta-4}$ and the cop plays \CSB, then we also get 
\[
\Prob[\max\{n\colon T_c(n) \leq \tim\} \geq (\mu(p_c^t)-\epsilon)\tim] > 1 - a e^{-b\tim}.
\]
\end{lemma}

\begin{proof}
We only prove the bounds for the robber. The bounds for the cop are proved exactly the same way.

Consider the standard coupling of a game $(\C_{\tim}, \R_{\tim})$ where the robber plays \RSB\  and a game $(\bar\C_{\tim}, \bar\R_{\tim})$ where the robber plays any other strategy, where neighbors of all vertices are ordered according to their distance from $B$. It is easy to show (for instance by induction on $\tim$) that $d(\R_{\tim},B) \leq d(\bar\R_{\tim},B)$ for all $\tim$, and hence $T_r(n) \leq \bar T_r(n)$, where $T_r(n)$ and $\bar T_r(n)$ are the time steps at which the $n$-th robber move along an edge of $B$ occurs in these games, respectively. This implies that if we can prove both bounds for a robber playing \RSB, then we will also have proved the claimed bound for arbitrary strategies. 

For the remainder of the proof, assume that the robber plays \RSB. Denote by $N_r(\tim)$ the number of robber moves up to time $\tim$. Since the probability of a robber move at any given time is $\frac 12$, by Theorem \ref{thm:concentration} for every $\epsilon > 0$ there is a positive constant $c$ such that
\begin{equation}
    \Prob[\tim (1-\epsilon) <2N_r(\tim) <\tim (1+\epsilon)] > 1-e^{-c\tim}. \label{eq:robbermovesconcentration}
\end{equation} 

Let $d_{\tim} = d(\R_{\tim},B)$. Keeping track of $d_{\tim}$ only at robber moves gives a random walk with the following transition probabilities.
\begin{align*}
    \Prob[d_{\tim} = d_{\tim-1}+1 \mid M_{\tim} \in \{\mathrm{rs},\mathrm{rt}\} \land d_{\tim-1} >0] &= 2 p_r^t (\delta-1)/\delta,\\
    \Prob[d_{\tim} = d_{\tim-1}-1 \mid M_{\tim} \in \{\mathrm{rs},\mathrm{rt}\} \land d_{\tim-1} >0] &= 2 p_r^s + 2 p_r^t/\delta,\\
    \Prob[d_{\tim} = d_{\tim-1}+1 \mid M_{\tim} \in \{\mathrm{rs},\mathrm{rt}\} \land d_{\tim-1} =0] &= 2 p_r^t/\Delta,\\
    \Prob[d_{\tim} = d_{\tim-1} \mid M_{\tim} \in \{\mathrm{rs},\mathrm{rt}\} \land d_{\tim-1} =0] &= 2 p_r^s + 2 p_r^t(\Delta-1)/\Delta.
\end{align*}
If $p_r^t < \frac{\delta}{4\delta-4}$, then the above transition probabilities together with \eqref{eq:robbermovesconcentration} and Lemma \ref{lem:positiverecurrentonN}, for every $\epsilon>0$ show that there are constants $a$ and $b$ such that
\[
    \Prob[(\bar\mu-\epsilon)\tim < 2 \max\{n \colon T_r(n) < \tim\} < (\bar\mu+\epsilon)\tim] > 1 - a e^{-b\tim},
\]
where
\begin{align*}
    \bar\mu = \frac{((2 p_r^s + 2 p_r^t/\delta) - (2 p_r^t (\delta-1)/\delta))(1-2 p_r^t/\Delta)}{2 p_r^t/\Delta + (2 p_r^s + 2 p_r^t/\delta) - (2 p_r^t (\delta-1)/\delta)}.
\end{align*}
It is easy to see (using $p_c^r = 1/2 - p_c^t$) that $\bar\mu = 2 \mu(p_r^t)$, thereby proving the claimed bounds for the robber in case $p_r^t < \frac{\delta}{4\delta-4}$. 

Observe that $\mu(p)$ is continuous in a neighbourhood of $\frac{\delta}{4\delta-4}$ and that $\mu(\frac{\delta}{4\delta-4}) =0$. Further note that for $p_r^t < \frac{\delta}{4\delta-4}$ and $\epsilon' > 0$ it is a valid robber strategy to make a random move with probability $\epsilon'$ and follow \RSB\ with probability $1-\epsilon'$. If $p_r^t < \frac{\delta}{4\delta-4}$ is large enough that $\mu(p_r^t) < \epsilon$, then it follows from the first part that  
\[
    \Prob[\max\{n\colon T_r(n) \leq \tim\} \leq  2\epsilon \tim] > 1 - a e^{-b\tim}
\]
for a robber playing this strategy. Finally note that for some suitable $\epsilon'$, the probability distribution of robber moves for a robber playing this strategy with tipsiness parameter $p_r^t$ is the same as the distribution for a robber playing \RSB\ with tipsiness parameter $\frac{\delta}{4\delta-4}$. This shows that the claimed bound also holds for $p_r^t = \frac{\delta}{4\delta-4}$.
\end{proof}

Recall that $\C_{\tim}'$ and $\R_{\tim}'$ are projections of the cop's and robber's position to the base tree $B$.
If the $\C_{\tim}'$ is never equal to $\R_{\tim}'$, then there is a positive probability that the cop never reaches the same small tree as the robber, in which case the robber wins the game.
The next result shows that this is essentially the only obstruction to \CSB\ winning against \RSB. 
We remark that a refinement of the argument given below shows that \CS\ is also winning against \RSB\ in the second case. Indeed, we conjecture that if the cop has a strategy to win the game on a tree, then they can always win the game by playing \CS.

\begin{theorem}
\label{thm:RSB} 
Consider the game on $X(\Delta,\delta)$.
Assume that $p_r^t \leq \frac{\delta}{4\delta-4}$ and $p_c^t \leq \min \left \{ \frac{\delta}{4\delta-4}, \frac{\Delta}{4\Delta-6} \right \} $. If 

\begin{align}\label{second inequality for prop1}
    \frac{\left(1-\frac{4 \delta -4}{\delta }p_r^t\right) \left(1 - \frac{6}{\Delta} p_r^t\right)}{1- \left(\frac{4 \delta -4}{\delta }-\frac{2}{\Delta }\right)p_r^t} -  \frac{\left(1-\frac{4 \delta -4}{\delta }p_c^t\right) \left(1 - \frac{4\Delta - 6}{\Delta} p_c^t\right)}{1- \left(\frac{4 \delta -4}{\delta }-\frac{2}{\Delta }\right)p_c^t} > 0
\end{align}
then \RSB is winning against any cop strategy.

If the reverse inequality holds:
\begin{align}\label{inequality for prop1}
    %p_r^t > \frac{p_c^t(3-5p_c^t)}{2-3p_c^t}
    \frac{\left(1-\frac{4 \delta -4}{\delta }p_r^t\right) \left(1 - \frac{6}{\Delta} p_r^t\right)}{1- \left(\frac{4 \delta -4}{\delta }-\frac{2}{\Delta }\right)p_r^t} -  \frac{\left(1-\frac{4 \delta -4}{\delta }p_c^t\right) \left(1 - \frac{4\Delta - 6}{\Delta} p_c^t\right)}{1- \left(\frac{4 \delta -4}{\delta }-\frac{2}{\Delta }\right)p_c^t} < 0
\end{align}
then \CSB\ is winning against \RSB; in particular \RSB\ is not winning against every cop strategy.
\end{theorem}

\begin{proof}
We split the proof in the following cases:
\begin{enumerate}[itemindent=1em,label=Case \arabic*:,ref=\arabic*]
    \item  \label{case one} $p_r^t < \frac{\delta}{4\delta-4}$ and $p_c^t\leq \min \left \{ \frac{\delta}{4\delta-4}, \frac{\Delta}{4\Delta-6} \right \}$, and \eqref{second inequality for prop1} holds,
    \item \label{case two}$p_r^t \leq \frac{\delta}{4\delta-4}$ and $p_c^t< \min \left \{ \frac{\delta}{4\delta-4}, \frac{\Delta}{4\Delta-6} \right \}$, and \eqref{inequality for prop1} holds,
    %\item \label{case three}$p_r^t = \frac{\delta}{4\delta-4}$ and $p_c^t < \min\left \{ \frac{\delta}{4\delta-4},\frac{\Delta}{4\Delta-6} \right \} $, in this case inequality \eqref{inequality for prop1} holds, and
    %\item \label{case four}$p_r^t < \frac{\delta}{4\delta-4}$ and $p_c^t= \min \left \{\frac{\delta}{4\delta-4},\frac{\Delta}{4\Delta-6} \right \}$, in this case inequality \eqref{second inequality for prop1} holds.
\end{enumerate}
Note that this indeed covers all necessary cases, since $p_r^t = \frac{\delta}{4\delta-4}$ implies that \eqref{second inequality for prop1} is not satisfied, and $p_c^t = \min \left \{ \frac{\delta}{4\delta-4},\frac{\Delta}{4\Delta-6} \right \}$ implies that \eqref{inequality for prop1} is not satisfied. %we do not need to consider the case $p_r^t = \frac{\delta}{4\delta-4}$ and $p_c^t = \min \left \{ \frac{\delta}{4\delta-4},\frac{\Delta}{4\Delta-6} \right \}$, because in this case neither of the two inequalities \eqref{second inequality for prop1} and \eqref{inequality for prop1} is satisfied.

Assume that we are in Case~\ref{case one}, and thus the bound in~\eqref{second inequality for prop1} holds. Assume that the robber plays strategy \RSB, and the cop plays an arbitrary strategy. We claim that for any $\epsilon' > 0$ we can find $\epsilon > 0$ and positive constants $a,b$ such that the probability that the following chain of inequalities is violated for any given $\tim$ is bounded above by $ae^{-b\tim}$, where $\mu(p)$ is defined as in Lemma \ref{lemma:stationary_time}:

\begin{align}\label{ineq1}
    \frac{Y_{\tim}}{\tim} &\geq 
    \left(\sum_{n:  T_r(n) < \tim} \frac{\tilde F_r(n)}{\tim}\right) + \left(\sum_{n:  T_c(n) < \tim} \frac{\tilde F_c(n)}{\tim}\right)\\\label{ineq2}
    &\geq \frac{\max\{n : T_r(n) \leq \tim\}- \epsilon \tim}{\tim} \left( \E[\tilde F_r(1)] - \epsilon\right) +\frac{\max\{n : T_c(n) \leq \tim\}+ \epsilon \tim}{\tim} \left( \E[\tilde F_c(1)] - \epsilon\right)\\\label{ineq3}
    &\geq \left(\mu(p_r^t)-2\epsilon\right) \left(\E[\tilde F_r(1)] - \epsilon\right) + \left(\mu(p_c^t)+2\epsilon\right) \left(\E[\tilde F_c(1)]-\epsilon\right)\\\label{ineq4}
    &\geq \mu(p_r^t)\E[\tilde F_r(1)] + \mu(p_c^t)\E[\tilde F_c(1)]- \epsilon'\\\label{equality}
    &=\frac 12 \cdot \frac{\left(1-\frac{4 \delta -4}{\delta }p_r^t\right) \left(1 - \frac{6}{\Delta} p_r^t\right)}{1- \left(\frac{4 \delta -4}{\delta }-\frac{2}{\Delta }\right)p_r^t} - \frac 12 \cdot \frac{\left(1-\frac{4 \delta -4}{\delta }p_c^t\right) \left(1 - \frac{4\Delta - 6}{\Delta} p_c^t\right)}{1- \left(\frac{4 \delta -4}{\delta }-\frac{2}{\Delta }\right)p_c^t}- \epsilon'.%\\
\end{align}
Clearly it is enough to show that each individual inequality in this chain is satisfied with probability $1-ae^{-b\tim}$ at at least for suitable constants $a$ and $b$. For inequality in \eqref{ineq1}, this follows from the fact that $F_r(n) \geq \tilde F_r(n)$, and $F_c(n) \geq \tilde F_c(n)$. For the second inequality \eqref{ineq2} note that $\E[\tilde F_r(1)] > 0$ and $\E[\tilde F_c(1)] \leq 0$ due to the assumptions of Case \ref{case one}. Hence \eqref{ineq2} is trivially true when $\max\{n : T_r(n) \leq \tim\}$ and $\max\{n : T_c(n) \leq \tim\}$ are smaller than $\epsilon \tim/(\E[\tilde F_r(1)] - \epsilon)$ and $\epsilon \tim/(\E[\tilde F_c(1)] - \epsilon)$, respectively. If one or both of them are larger, then we can apply Theorem \ref{thm:concentration}. 
Inequality \eqref{ineq3} follows from Lemma \ref{lemma:stationary_time}. For the last inequality \eqref{ineq4} we note that none of the coefficients are unbounded, and thus for given $\epsilon'$ any small enough $\epsilon$ works.

If $\epsilon'$ is small enough, then the expression in \eqref{equality} is larger than $0$.
Since $\sum_{\tim \geq 1} ae^{-b\tim} < \infty$, we conclude that there is some $\tim_0$ such that the probability that $Y_{\tim}>0$ for every $\tim > \tim_0$ is positive. If the distance between $\R_{0}'$ and $\C_{0}'$ is larger than $\tim_0$, then this implies that there is a positive probability that $\R_{\tim}'$ and $\C_{\tim}'$ never coincide, and thus the robber wins the game with positive probability.

Next consider Case~\ref{case two}. For $\tim>0$, we denote by $K(\tim)$ the number of $\timtwo \leq \tim$ such that $\C_{\timtwo}' = \R_{\timtwo}'$. If the robber plays \RSB, then $F_r(n) > \tilde F_r(n)$ only if $\C_{\tim}' = \R_{\tim}'$ at time $\tim = T_r(n)$, and in this case $F_r(n) - \tilde F_r(n) \leq 2$. Similarly, if the cop plays \CSB, then $F_c(n) > \tilde F_c(n)$ only if $\C_{\tim}' = \R_{\tim}'$ at time $\tim = T_c(n)$, and in this case $F_c(n) - \tilde F_c(n) \leq 2$. In particular, 
\[
 \frac{Y_{\tim}}{\tim} \leq 
    \left(\sum_{n:  T_r(n) < \tim} \frac{\tilde F_r(n)}{\tim}\right)  + \left(\sum_{n:  T_c(n) < \tim}\frac{\tilde F_c(n)}{\tim} \right)+ \frac{2 K(\tim)}{\tim},\\
\]
and a similar calculation as above yields that for every $\epsilon' >0$ there are positive constants $a$ and $b$ such that the inequality
\begin{align}\label{ineqonYs}
 \frac{Y_{\tim}}{\tim} \leq 
    \frac 12 \cdot \frac{\left(1-\frac{4 \delta -4}{\delta }p_r^t\right) \left(1 - \frac{6}{\Delta} p_r^t\right)}{1- \left(\frac{4 \delta -4}{\delta }-\frac{2}{\Delta }\right)p_r^t} - \frac 12 \cdot \frac{\left(1-\frac{4 \delta -4}{\delta }p_c^t\right) \left(1 - \frac{4\Delta - 6}{\Delta} p_c^t\right)}{1- \left(\frac{4 \delta -4}{\delta }-\frac{2}{\Delta }\right)p_c^t} + \epsilon' + \frac{2 K(\tim)}{\tim}
\end{align}
is satisfied with probability $1-ae^{-b \tim}$. Since the distance between $\C_{\tim}'$ and $\R_{\tim}'$ is never negative, there is a constant lower bound on $Y_{\tim}$. Hence choosing $\epsilon'$ sufficiently small, we conclude that there is $\epsilon > 0$ such that almost surely
\begin{equation}
    \liminf \frac{K(\tim)}{\tim} \geq \epsilon. \label{eq:kdensity}
\end{equation}

We note that there are constants $p$ and $q$ such that 
\[
    \Prob[d(\C_{\tim},\C'_{\tim}) > d(\C_{\tim+1},\C'_{\tim+1}) \mid \C_{\tim} \neq \C'_{\tim}] \leq p < q \leq \Prob[d(\C_{\tim},\C'_{\tim}) < d(\C_{\tim+1},\C'_{\tim+1})\mid \C_{\tim} \neq \C'_{\tim}].
\]
Hence by Propositions \ref{prp:inhomogeneous_recurrent} and \ref{prp:occupationmeasure}, for every $\epsilon > 0$ there is some $n_0$ such that almost surely $\liminf \frac{1}{\tim} |\{\timtwo < \tim \colon d(\C_{\timtwo},\C'_{\timtwo}) \leq n_0\}| > 1-\epsilon$.
This, together with \eqref{eq:kdensity} implies that there is some $n_0$ for which 
\[  
    \Prob[\forall \tim \exists \timtwo > \tim \colon \C'_{\timtwo} = \R'_{\timtwo} \text{ and }d(\C_{\timtwo},\C'_{\timtwo}) \leq n_0] = 1.
\]

We hence need to analyse the game provided that the starting positions $(\C_0,\R_0) = (c,r)$ satisfy $\C'_{0} = \R'_{0}$ and $d(\C_{0},\C'_{0}) \leq n_0$. Denote by $N = \min\{\tim>0 \colon d(\R_{\tim}, \R_{\tim}') \leq n_0\}$ and note that $N$ is almost surely finite (because every move is at least as likely to decrease the distance between $\R_{\tim}$ and $\R_{\tim}'$ as it is to increase this distance).

We claim that there is a constant $p$ such that
\[
    \Prob_{(c,r)}[d(\C_N,\R'_N) \leq n_0]\geq p.
\]
To prove this claim note that cop moves and robber moves up to time $N$ are independent; sober moves of either players reduce the distance to $\R'_{\tim}$, which for $\tim < N$ only depends on the initial position. Let $D_{\tim} = d(\C_{\tim},\R_{0}')$. Note that we can find constants $p$ and $q$ such that conditional on $N \geq \tim$ we have
\[
\Prob_{(c,r)}[D_{\tim} =  D_{\tim-1} +1]\leq p < q \leq \Prob_{(c,r)}[D_{\tim} =  D_{\tim-1} -1].
\]
Because $D_{0} \leq n_0$, by Propositions \ref{prp:inhomogeneous_recurrent} and \ref{prp:occupationmeasure} there is a constant $p$ such that $\Prob_{(c,r)}[D_{\tim} \leq n_0] \geq p$ for every $\tim < N$. The value of $p$ does not depend on $r$, and we may assume (for instance by taking the minimum over all possible choices of $c$) that it does not depend on $c$ either. In particular $\Prob_{(c,r)}[D_{N} \leq n_0] \geq p$. 
Since 
\[
    \Prob_{(c,r)}[\exists \tim > N \colon \C'_{\tim} = \R'_{\tim} \text{ and }d(\C_{\tim},\C'_{\tim}) \leq n_0] =1,
\]
this implies that 
\[
    \Prob_{(c,r)}[\exists\tim \geq 0 \colon d(\C_{\tim},\R_{\tim}) \leq 2 n_0]=1,
\]
and thus 
\[
    \Prob[\forall \tim \exists\timtwo > \tim \colon d(\C_{\timtwo},\R_{\timtwo}) \leq 2 n_0]=1.
\]
Finally, every time $d(\C_{\tim},\R_{\tim}) < 2n_0$, the probability that the cop wins the game within the next $2n_0$ steps is at least $(p_c^t/\Delta)^{2n_0} >0$. Therefore the cop almost surely wins the game.
\end{proof}

\subsection{Comparison of the two strategies}

The aim of this section is to compare the two strategies \RSA\ and \RSB\ against one another. Throughout this section we will assume that $p_r^t < \frac 12$ because for $p_r^t = \frac 12$ all robber moves are random, and thus there will not be any difference between the two robber strategies.

We first show that (apart from isolated cases) we may assume that $0<p_c^t < \min \{ \frac{\delta}{4\delta-4},\frac{\Delta}{4\Delta-6} \}$ and that $0 < p_r^t < \frac{\delta}{4\delta-4}$. Proposition \ref{prp:coptootipsy} tells us that if $p_c^t > \min\{\frac{\delta}{4\delta-4},\frac{\Delta}{4\Delta-6}\}$ then the robber wins if they play each of the two strategies \RSA\ or \RSB. Our first result tells us that if $p_r^t > \frac{\delta}{4\delta-4}$ then the robber should always choose strategy \RSA\ over strategy \RSB.

\begin{proposition}
\label{prp:compare-drunkrobber}
If $p_r^t > \frac{\delta}{4\delta-4}$ and the cop has a winning strategy against a robber playing \RSA, then they also have a winning strategy against a robber playing \RSB; in this case \CS\ is winning against both \RSA\ and \RSB.
\end{proposition}

\begin{proof}
If the cop has a winning strategy against a robber playing \RSA, then by Theorem \ref{thm:robbersober} we know that $p_c^t \leq \min\{\frac{p_r^t}{\delta-1},\frac{\Delta}{4\Delta-6}\}$, which means that \CS\ is winning against \RSB\ provided that $p_r^t > \frac{\delta}{4\delta-4}$.
\end{proof}

Hence we can assume that $0<p_c^t \leq \min \{\frac{\delta}{4\delta-4},\frac{\Delta}{4\Delta-6} \}$, and that $0< p_r^t \leq \frac{\delta}{4\delta-4}$. The next proposition deals with the cases where equality holds in one of the two bounds.

\begin{proposition} 
\label{prp:compare-boundary}
Consider the game on $X(\Delta,\delta)$, and assume that $\Delta \geq 4$.
\begin{enumerate}
    \item If $p_c^t = 0$, then \CS\ is always a winning strategy for the cop.
    \item If $p_r^t = 0$ and $p_c^t \neq 0$, then both \RSA\ and \RSB\ are winning robber strategies.
    \item If $p_r^t = \frac{\delta}{4\delta-4}$ and \RSB\ is winning against every cop strategy, then so is \RSA.
\end{enumerate}
\end{proposition}

\begin{proof}
The first two statements follow from Proposition \ref{prp:inhomogeneous_recurrent} by observing that the probability that the distance between the cop and the robber increases is at most $\frac 12$ in the first case, and at least $\frac 12 + \epsilon$ for some $\epsilon > 0$ in the second. For the third statement, recall that we are working under the assumption that $p_r^t< \frac 12$ and hence $\delta \geq 3$. If \RSA\ is not winning, then $p_c^t \leq \frac{\delta}{4(\delta-1)^2}$ by Theorem \ref{thm:robbersober}. In particular, $p_c^t < \frac 14 \leq \min \{ \frac{\delta}{4\delta-4}, \frac{\Delta}{4 \Delta-6} \}$ and thus \RSB\ is not winning against \CSB\ by Theorem \ref{thm:RSB}.
%The last statement is a direct consequence of Theorems \ref{thm:robbersober} and  \ref{thm:RSB}.
\end{proof}

The next lemma puts the winning conditions from Theorem \ref{thm:RSB} for a robber playing \RSB\ in a more useful form.

\begin{lemma}
\label{lem:threshold-function}Let $\Delta\geq 4$.
There is a function 
\[f\colon \left[0,\frac{\delta}{4\delta-4}\right] \to \left[0, \min \left \{ \frac{\delta}{4\delta-4},\frac{\Delta}{4\Delta-6}\right \} \right]
\] 
such that \RSB\ is winning against any cop strategy if $f(p_r^t) < p_c^t \leq \min\{ \frac{\delta}{4\delta-4},\frac{\Delta}{4\Delta-6}\}$, and \CSB\ is winning against \RSB\ if $0\leq p_c^t < f(p_r^t)$.

This function is monotonically increasing and strictly convex, $f(0)=0$, $f(\frac{\delta}{4\delta-4}) =  \min \left\{ \frac{\delta}{4\delta-4},\frac{\Delta}{4\Delta-6} \right \}$, and $f'(0) = \frac{2}{\Delta-1}$.
\end{lemma}
\begin{proof}
If we let $G(x) = \frac{(1-ax)(1-bx)}{(1-cx)}$ then $G'(x) = -\frac{b (1-a x)}{1-c x}-\frac{(a-c) (1-b x)}{(1-c x)^2}$ and $G''(x) = \frac{2 (a-c) (b-c)}{(1-c x)^3}$ provided that $x \neq 1/c$. 

For fixed $x$, $\delta$, and $\Delta$ this implies that the expression
\[
F(x,y) = \frac{\left(1-\frac{4 \delta -4}{\delta }x\right) \left(1 - \frac{6}{\Delta} x\right)}{1- \left(\frac{4 \delta -4}{\delta }-\frac{2}{\Delta }\right)x} -  \frac{\left(1-\frac{4 \delta -4}{\delta }y\right) \left(1 - \frac{4\Delta - 6}{\Delta} y\right)}{1- \left(\frac{4 \delta -4}{\delta }-\frac{2}{\Delta }\right)y}
\]
is strictly increasing for $y$ in the interval $\left [0,\min \left \{ \frac{\delta}{4\delta-4}, \frac{\Delta}{4\Delta-6} \right \} \right ]$. It is not hard to check that $F(x,0) \leq 0$ and $F\left(x, \min \left \{ \frac{\delta}{4\delta-4}, \frac{\Delta}{4\Delta-6} \right \}\right) \geq 0$ for $x \in \left[0,\frac{\delta}{4\delta-4}\right]$. Hence the equation $F(x,f(x)) = 0$ defines a unique function $f(x)$. By Theorem \ref{thm:RSB} this function has the property that if $f(p_r^t) < p_c^t \leq \min\{ \frac{\delta}{4\delta-4},\frac{\Delta}{4\Delta-6}\}$, then \RSB\ is winning against any cop strategy, and if $0\leq p_c^t < f(p_r^t)$, then \CSB\ is winning against \RSB.

To see that $f$ is monotonically increasing, recall that 
\[
    f'(x_0) = -\frac{F_x}{F_y},
\]
where the partial derivatives are evaluated at the point $(x_0,f(x_0))$.
The partial derivatives have the form $\pm G'$ from above for suitable values of $a$, $b$, and $c$ (depending on $\delta$ and $\Delta$), and it is not hard to see that $F_x(x_0,y_0)$ is positive and $F_y(x_0,y_0)$ is negative for all pairs 
\[(x_0,y_0) \in \left(0,\frac{\delta}{4\delta-4}\right) \times \left(0, \min \left \{ \frac{\delta}{4\delta-4},\frac{\Delta}{4\Delta-6}\right \} \right).
\]

To show convexity, recall that
\[ 
    f''(x_0) = \frac{-F_y^2F_{xx} + 2 F_xF_y F_{xy} -F_x^2F_{yy}}{F_y^3}.
\]
Note that the mixed second derivative $F_{xy}$ is always zero and that the denominator $F_y^3$ is always negative. Moreover $F_{xx}$ and $F_{yy}$ have the form $\pm G''$ from above for suitable values of $a$, $b$, and $c$, and again it is not hard to see that for all values of $(x_0,y_0)$ in our range we have $F_{xx}(x_0,y_0) \geq 0$ (with equality if $\delta = 2$ and $\Delta = 4$, here we are using the assumption that $\Delta \geq 4$),  and $F_{yy}(x_0,y_0) > 0$.

Finally note that 
\[
F(0,0) = F\left(\frac{\delta}{4\delta-4}, \min \left \{ \frac{\delta}{4\delta-4},\frac{\Delta}{4\Delta-6}\right \} \right) = 0,
\]
and thus $f(0)=0$, and $f(\frac{\delta}{4\delta-4}) =  \min \left \{ \frac{\delta}{4\delta-4},\frac{\Delta}{4\Delta-6} \right \} $ as claimed. Moreover, evaluating $f'(x_0) =  -\frac{F_x}{F_y}$ at $x_0 = f(x_0) = 0$ gives $\frac{2}{\Delta-1}$.
\end{proof}

We are now ready to prove Theorem \ref{theorem:strategy_summary}. Before we do so, let us recall the statement of this theorem.

\begingroup
\def\thetheorem{\ref{theorem:strategy_summary}}
\begin{theorem}
Consider the game on $X(\Delta,\delta)$ with $\Delta \geq 4$ and assume that $p_c^s+p_c^t = p_r^s + p_r^t= 1/2$. There is some $p_0 \in [0,\frac{\delta}{4(\delta-1)}]$ such that
\begin{enumerate}
    \item if $p_r^t \leq p_0$ and \RSA\ is winning against every cop strategy for some fixed $p_c^t$, then so is \RSB, and
    \item if $p_r^t > p_0$ and \RSB\ is winning against every cop strategy for some fixed $p_c^t$, then so is \RSA.
\end{enumerate}
If $\delta=2$, then $p_0 = \frac 12$. If $2 \delta \geq \Delta+1$, then $p_0 = 0$. Otherwise the value of $p_0$ lies strictly between $0$ and $\frac 12$ and can be determined by solving a quadratic equation whose coefficients depend on $\delta$ and $\Delta$.
\end{theorem}
\addtocounter{theorem}{-1}
\endgroup

\begin{proof}
First note that if $p_r^t=0$, then by Proposition \ref{prp:compare-boundary} the strategies \RSA\ and \RSB\ are both winning against any cop strategy for $p_c^t > 0$ and both losing against \CS\ for $p_c^t = 0$. If $p_r^t > \frac{\delta}{4\delta-4}$ and \RSB\ is winning for a given value $p_c^t$, then by Proposition \ref{prp:compare-drunkrobber} so is \RSA. In particular, if the theorem is true for $p_r^t \in (0,\frac{\delta}{4\delta-4}]$, then it is true for $p_r^t \in [0,\frac 12]$.

Throughout the rest of this proof, assume that $0 < p_r^t \leq \frac{\delta}{4\delta-4}$. Denote by $f(x)$ the function provided by Lemma \ref{lem:threshold-function}, and let $g(x) = \frac{x}{\delta-1}$. Observe that $f(0) = g(0) = 0$. By Theorem \ref{thm:robbersober}, \RSA\ is winning against every possible cop strategy if and only if $p_c^t > g(p_r^t)$, otherwise it loses against \CS. By Lemma \ref{lem:threshold-function}, \RSB\ is winning against every cop strategy if $p_c^t > f(p_r^t)$ and losing against \CSB\ if $p_c^t < f(p_r^t)$.

If $\delta=2$, then $\frac{\delta}{4\delta-4} = \frac 12$. By Lemma \ref{lem:threshold-function} we have that $f(\frac 12) = \min \{ \frac 12, \frac{\Delta}{4\Delta-6} \}$ and since $\Delta \geq 4$ this is strictly less than $g(\frac 12) = \frac 12$. By convexity of $f$ this implies that $f(x) < g(x)$ for $x \in (0,\frac 12)$ and thus if \RSA\ is winning against every cop strategy for a given value $p_c^t$, then so is \RSB.

Now assume that $\delta \geq 3$. In this case $f(\frac{\delta}{4\delta-4}) > \frac 14 \geq g(\frac{\delta}{4\delta-4})$. If $2\delta \geq \Delta +1$, then $f'(0) \geq g'(0)$ and strict convexity of $f$ implies that $f(x) > g(x)$ for every $x \in (0,\frac{\delta}{4\delta-4}]$. This in particular implies that if $p_r^t > 0$ and \RSB\ is winning against every cop strategy for a given value $p_c^t$, then so is \RSA. Finally, if $2\delta < \Delta +1$, then $f'(0) < g'(0)$, and consequently there is a unique $t_0 \in (0,\frac{\delta}{4\delta-4}]$ such that $f(x) \leq g(x)$ for $x \leq t_0$ and $f(x) > g(x)$ for $x > t_0$. This implies that if $p_r^t \leq t_0$ and \RSA\ is winning against every cop strategy for a given value $p_c^t$, then so is \RSB, and if $p_r^t > t_0$ and \RSB\ is winning against every cop strategy for a given value $p_c^t$, then so is \RSA.

Recall that $f(x)$ is the unique solution to the equation $F(x,f(x)) = 0$, where 
\[
F(x,y) = \frac{\left(1-\frac{4 \delta -4}{\delta }x\right) \left(1 - \frac{6}{\Delta} x\right)}{1- \left(\frac{4 \delta -4}{\delta }-\frac{2}{\Delta }\right)x} -  \frac{\left(1-\frac{4 \delta -4}{\delta }y\right) \left(1 - \frac{4\Delta - 6}{\Delta} y\right)}{1- \left(\frac{4 \delta -4}{\delta }-\frac{2}{\Delta }\right)y}.
\]
Since $t_0$ is a solution of $f(x) = g(x)$, we can find its value by solving the equation $F(x,\frac{x}{\delta-1}) = 0$ for $x$, which gives a quadratic equation in $x$ as claimed.
\end{proof}

In Figure \ref{examples for theorem 4.2} we compare 
the efficacy of the robber's dualing strategies on $X(\Delta, \delta)$ for $\Delta = 7$ and $\delta \in \{2,3,4,5\}$. The dashed lines are at $p_c^t = \frac{\delta}{4\delta-4}$ and $p_r^t = \frac{\delta}{4\delta-4}$, respectively; the dotted line is at $p_c^t = \frac{\Delta}{4\Delta-6}$. The red and blue lines are the graphs of the functions $f$ and $g$ from the proof of Theorem \ref{theorem:strategy_summary}; if for $p_r^t < \frac{\delta}{4\delta-4}$ the point $(p_r^t,p_c^t)$ lies above the blue line, then strategy \RSA\ is winning against every cop strategy, and if it lies above the red line \RSB\ is winning against every cop strategy.

\begin{figure}[h!]
\centering
\begin{tikzpicture}[scale=.8]
\begin{axis}
[
  title ={$\Delta = 7$, $\delta = 2$},
  axis lines*=center,
  xtick={0,.1,.2,...,.5},
  ytick={0,.1,.2,...,.5},
  tick align=outside,
  yticklabel style={rotate=90},
  xlabel={$p_r^t$},
  x label style={at={(axis description cs:0.5,0)}},
  y label style={at={(axis description cs:.05,.5)}},
  ylabel={$p_c^t$},
  ymax=.5, ymin=0,
  xmax=.5, xmin=0
]
\addplot[dashed, samples=2,black] coordinates {(.5,0)(.5,.5)};
\addplot[dashed, samples=2,black] coordinates {(0,.5)(.5,.5)};
\addplot[dotted, samples=2,black] coordinates {(0,7/22)(7/22,7/22)};
\addplot [
    domain=0:0.5, 
    samples=100, 
    color=red,
]
{(-21 + 24*x + 18*x^2+sqrt(441-2086*x+3285*x^2-1908*x^3+324*x^4))/(11*(-7+12*x))};
\addplot [
    domain=0:7/22, 
    samples=2, 
    color=blue,
]
{x};
\addplot [
    domain=7/22:.5, 
    samples=2, 
    color=blue,
]
{7/22};
\end{axis}
\end{tikzpicture}%
\qquad
\begin{tikzpicture}[scale=.8]
\begin{axis}
[
  title ={$\Delta = 7$, $\delta = 3$},
  axis lines*=center,
  xtick={0,.1,.2,...,.5},
  ytick={0,.1,.2,...,.5},
  tick align=outside,
  yticklabel style={rotate=90},
  xlabel={$p_r^t$},
  x label style={at={(axis description cs:0.5,0)}},
  y label style={at={(axis description cs:.05,.5)}},
  ylabel={$p_c^t$},
  ymax=.5, ymin=0,
  xmax=.5, xmin=0
]
\addplot[dashed, samples=2,black] coordinates {(3/8,0)(3/8,.5)};
\addplot[dashed, samples=2,black] coordinates {(0,3/8)(.5,3/8)};
\addplot[dotted, samples=2,black] coordinates {(0,7/22)(.5,7/22)};
\addplot [
    domain=0:3/8, 
    samples=100, 
    color=red,
]
{(3*(-63+100*x+100*x^2+sqrt(3969-25536*x+54072*x^2-41600*x^3+10000*x^4)))/(44*(-21+50*x))};

{7/22};
\addplot [
    domain=0:0.5, 
    samples=100, 
    color=blue,
]
{x/2};
\end{axis}
\end{tikzpicture}

\vspace{2ex}
\begin{tikzpicture}[scale=.8]
\begin{axis}
[
  title ={$\Delta = 7$, $\delta = 4$},
  axis lines*=center,
  xtick={0,.1,.2,...,.5},
  ytick={0,.1,.2,...,.5},
  tick align=outside,
  yticklabel style={rotate=90},
  xlabel={$p_r^t$},
  x label style={at={(axis description cs:0.5,0)}},
  y label style={at={(axis description cs:.05,.5)}},
  ylabel={$p_c^t$},
  ymax=.5, ymin=0,
  xmax=.5, xmin=0
]
\addplot[dashed, samples=2,black] coordinates {(1/3,0)(1/3,.5)};
\addplot[dashed, samples=2,black] coordinates {(0,1/3)(.5,1/3)};
\addplot[dotted, samples=2,black] coordinates {(0,7/22)(.5,7/22)};
\addplot [
    domain=0:1/3, 
    samples=100, 
    color=red,
]
{(-84+152*x+171*x^2+sqrt(7056-51408*x+122812*x^2-106020*x^3+29241*x^4))/(2*(-231+627*x))};

\addplot [
    domain=0:0.5, 
    samples=100, 
    color=blue,
]
{x/3};
\end{axis}
\end{tikzpicture}%
\qquad
\begin{tikzpicture}[scale=.8]
\begin{axis}
[
  title ={$\Delta = 7$, $\delta = 5$},
  axis lines*=center,
  xtick={0,.1,.2,...,.5},
  ytick={0,.1,.2,...,.5},
  tick align=outside,
  yticklabel style={rotate=90},
  xlabel={$p_r^t$},
  x label style={at={(axis description cs:0.5,0)}},
  y label style={at={(axis description cs:.05,.5)}},
  ylabel={$p_c^t$},
  ymax=.5, ymin=0,
  xmax=.5, xmin=0
]
\addplot[dashed, samples=2,black] coordinates {(5/16,0)(5/16,.5)};
\addplot[dashed, samples=2,black] coordinates {(0,5/16)(.5,5/16)};
\addplot[dotted, samples=2,black] coordinates {(0,7/22)(.5,7/22)};
\addplot [
    domain=0:5/16, 
    samples=100, 
    color=red,
]
{(-525 + 1020*x + 1224*x^2 + sqrt(
 275625 - 2149000*x + 5484000*x^2 - 5042880*x^3 + 
  1498176*x^4))/(88*(-35 + 102*x))};
\addplot [
    domain=0:0.5, 
    samples=100, 
    color=blue,
]
{x/4};
\end{axis}
\end{tikzpicture}
\caption{
Comparing the efficacy of the robber's dualing strategies on $X(\Delta, \delta)$ for $\Delta = 7$ and $\delta \in \{2,3,4,5\}$.
% The dashed lines are at $p_c^t = \frac{\delta}{4\delta-4}$ and $p_r^t = \frac{\delta}{4\delta-4}$, respectively; the dotted line is at $p_c^t = \frac{\Delta}{4\Delta-6}$. The red and blue lines are the graphs of the functions $f$ and $g$ from the proof of Theorem \ref{theorem:strategy_summary}; if for $p_r^t < \frac{\delta}{4\delta-4}$ the point $(p_r^t,p_c^t)$ lies above the blue line, then strategy \RSA\ is winning against every cop strategy, and if it lies above the red line \RSB\ is winning against every cop strategy.
}
\label{examples for theorem 4.2}
\end{figure}
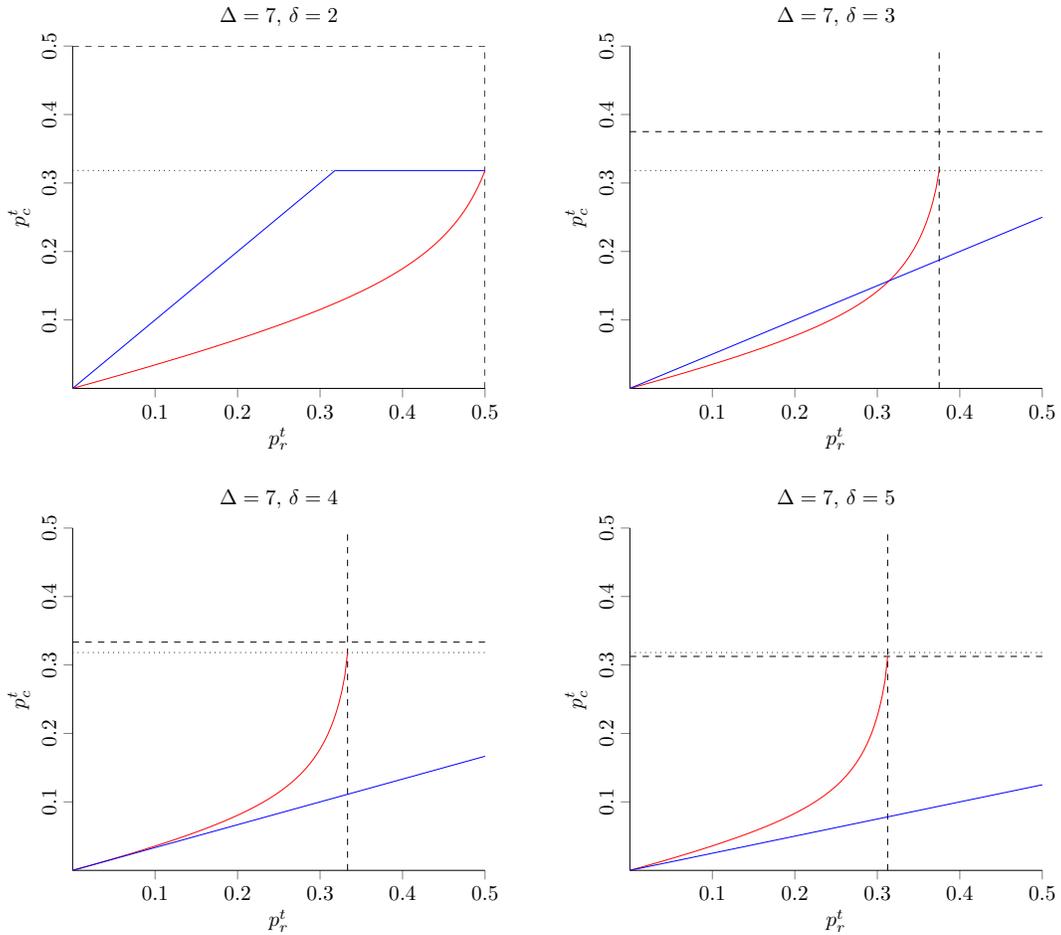
 
\section{Future Directions and Open Problems}\label{section:future}

It seems intuitively clear that on any tree, the cop's optimal strategy is to always move in the direction of the robber, but we were unable to provide a proof nor could we find a counterexample through the tools we used in our arguments.  Hence we pose the following problems, which might be solved via a different perspective or with a different set of mathematical tools than those used in this manuscript.

\begin{problem} \label{prob:51}
Prove or provide a counterexample to the following statement:  On any tree, the cop's optimal strategy is to always move in the direction of the robber.
\end{problem}

\begin{problem} \label{prob:52}
 Prove or provide a counterexample to the following statement: On any graph, if the cop has a winning strategy, then they also have a winning strategy which is ``monotone'' in the sense that the cop never purposefully increases the distance between them and the robber. 
\end{problem}

We conjecture that the statement in Problem~\ref{prob:51} is true, and there is likely a counterexample to the statement in Problem~\ref{prob:52}.

On the grid graph, we have shown that the cop's choice of strategy has an impact on whether they win when $c>r$, but that if the robber employs any strategy that increases the distance between them and the cop, then they are equally likely to win when $r>c$.  On the $X(\Delta,\delta)$ trees, the cop's intuitive strategy is to decrease the distance between them and the robber in each step, but the robber's strategy depends on conditions relating the proportion of sober cop and robber moves to $\Delta$ and $\delta$.
These examples lead us to ask, whether there are any graphs that restrict the strategies of both players at the same time. 

\begin{problem} Are there graphs where the cop and robber each have multiple strategies they could employ where each of these strategies is clearly better depending on specific conditions on $c$ and $r$?
\end{problem}

Recently, Songsuwan, Jiarasuksakun, Tanghanawatsakul, and Kaemawichanurat calculated the expected capture time of the game on the $n$-dimensional grid graph with the robber's tipsiness being $1$ and the cop's tipsiness being $0$ \cite{JKST21}. We think that it might be possible to refine our proof techniques to calculate the expected capture times for other tipsiness parameters as well.

\begin{problem}
Calculate the expected capture time on $n$-dimensional grid graphs for other tipsiness values.
\end{problem}

In this paper, we have studied one of several interesting notions of tipsiness. A particularly interesting generalization is a model where the cop and the robber become `more aware' of each others' presence when their distance is small, that is, they are more likely to play according to their strategies when the opposing player is nearby.

\begin{problem}
Study models where the tipsiness parameters increase with the distance between the cop and the robber. What can be said about optimal strategies and capture times? Do the results depend on how quickly the tipsiness parameters increase (compared to the dimension of the grid or to one another)?  
\end{problem}

 In this manuscript we focus on infinite graphs, but there are also many interesting questions regarding the game on finite graphs. A very natural question to ask is the ``cost of drunkenness": how does the capture time depend on the tipsiness of either player?
\begin{problem}
 On finite graphs, by how much can the expected capture time differ from the capture time in the deterministic game? What is the best strategy for a sober robber against a drunk cop on a finite tree, for instance a full binary tree of depth $k$? What about other tipsiness parameters for both players?
\end{problem}

\bibliography{Bibliography}
\bibliographystyle{plain}

\end{document}